%% file: Toric_varieties_and_unipotent_groups_v2.tex
\documentclass[a4paper,draft,reqno,12pt]{amsart}
\usepackage[english]{babel}
\usepackage{amsmath}
\usepackage{amssymb}
\usepackage{amscd}
\usepackage{amsthm}
\usepackage{euscript}
\usepackage{tikz}
\newtheorem{proposition}{Proposition}
\newtheorem{lemma}{Lemma}
\newtheorem{theorem}{Theorem}

\theoremstyle{definition}

\newtheorem{example}{Example}

\theoremstyle{remark}
\newtheorem {remark}{Remark}

\usepackage{tikz}
\usetikzlibrary{calc}

\DeclareMathOperator{\Spec}{Spec}

\DeclareMathOperator{\Aut}{Aut}

\def\BG{{\mathbb G}}
\def\BK{{\mathbb K}}

\def\BG{{\mathbb G}}
\def\BF{{\mathbb F}}

\def\BK{{\mathbb K}}

\def\BZ{{\mathbb Z}}

\def\BQ{{\mathbb Q}}

\def\BP{{\mathbb P}}
\def\BA{{\mathbb A}}

\def\Cl{\mathrm{Cl}}

\def\FR{\mathfrak{R}}

\def\WX{\widehat{X}}
\def\FS{\mathfrak{S}}

\title[Unipotent group actions on toric varieties with a finite number of orbits]{Toric varieties admitting an action of a unipotent group with a finite number of orbits}

\thanks{The paper was supported by the grant RSF 23-71-01100}

\author{Anton Shafarevich}
\email{shafarevich.a@gmail.com}
\address{
Lomonosov Moscow State University, Faculty of Mechanics and Mathematics, Department of Higher Algebra, Leninskie Gory 1, Moscow, 119991 Russia;
\linebreak
and
\linebreak
HSE University, Faculty of Computer Science, Pokrovsky Boulevard 11, Moscow, 109028, Russia}

\begin{document}
\maketitle

\begin{abstract}
We describe complete simplicial toric varieties on which a unipotent group acts with a finite number of orbits. We also provide a complete list of such varieties in the cases when the dimension is equal to 2 or the divisor class group is $\BZ$.
\end{abstract}

\section{Introduction}

Let $X$ be a complete toric variety with an acting algebraic torus $T$ over an algebraically closed field $\BK$ of characteristic zero. The group of regular automorphisms $\Aut(X)$ is well studied. It was proved in \cite{Dem} that when $X$ is smooth, this group is a linear algebraic group. Later, in \cite{COX2}, this result was generalized using a new technique to simplicial complete toric varieties.

In this work, we describe complete simplicial toric varieties on which a unipotent group acts with a finite number of orbits. Each unipotent subgroup in $\Aut(X)$ is contained in a maximal unipotent subgroup. If a unipotent group acts with a finite number of orbits, then a maximal unipotent group acts with a finite number of orbits. In turn, since $\Aut(X)$ is a linear algebraic group, all maximal unipotent subgroups are conjugated. Therefore, a complete toric variety admits an action of some unipotent group with a finite number of orbits if and only if a maximal unipotent group acts with a finite number of orbits. So it suffices to describe complete simplicial toric varieties on which a maximal unipotent group acts with a finite number of orbits.

It is clear that if a unipotent group acts with a finite number of orbits, then one of the orbits is open in the Zariski topology. Complete toric  varieties on which a unipotent group acts with an open orbit are called \emph{radiant} and were described in \cite{AR}.

If we denote by $N$ the lattice of one-parameter subgroups of $T$, then there is a fan $\Sigma$ in $N$ which corresponds to $X$. Let $\Sigma(1) = \{\rho_1,\ldots, \rho_d\}$ be the set of rays in $\Sigma$ and $n_1, \ldots, n_d\in N$ be the primitive vectors on these rays. Then each ray $\rho$ corresponds to a prime $T$-invariant divisor $D_{\rho}$ on $X$. We denote by $[D_{\rho}]$ the class of $D_{\rho}$ in the divisor class group~$\Cl(X)$.

When the group $\Cl(X)$ is free, the correspondence 
$$\{n_{1},\ldots, n_d\} \to \{[D_{\rho_1}], \ldots, [D_{\rho_d}]\}$$ is the linear Gale duality and was studied in \cite{BH}. For each cone $\sigma \in \Sigma$ we denote by $\sigma(1)$ the set of all rays in $\sigma.$ Then let $\Gamma(\sigma)$ be the submonoid in $\Cl(X)$ generated by the set $\{[D_{\rho}] \mid \rho \notin \sigma(1) \}.$ Monoids $\Gamma(\sigma)$ were used in \cite{Bazh} to describe orbits of the automorphism group of complete toric varieties. We say that an abelian finitely generated monoid $\Gamma$ is \emph{free} if it is isomorphic to the monoid $(\BZ_{\geq 0})^k$ for some $k$. Note that a finitely generated submonoid $\Gamma$ in $(\mathbb{Z}_{\geq 0})^m$ for some $m\in \BZ_{>0}$ is free if and only if the set of irreducible elements in $\Gamma$ is $\BQ$-linearly independent.

Our main result is the following theorem.

\begin{theorem}\label{mainth}

     Let $X$ be a complete simplicial toric variety and $\Sigma$ be the corresponding fan. Then a maximal unipotent subgroup in $\Aut(X)$ acts on $X$ with a finite number of orbits if and only if $X$ is radiant and monoids $\Gamma(\sigma)$ are free for all $\sigma \in \Sigma.$
\end{theorem}

We also classify all such varieties in the cases when  $X$ is a surface (Proposition \ref{dim2}) and when $\Cl(X)$ is isomorphic to $\BZ$ (Proposition \ref{projspace}).

Let us briefly describe the idea of the proof of Theorem \ref{mainth}. For a complete toric variety $X$ with a fan $\Sigma$ one can consider the algebra $R(X) = \BK[x_{\rho} \mid \rho \in \Sigma(1)]$ which is called the \emph{Cox ring} of the toric variety $X.$ There is an action of the group $G_X = \mathrm{Hom}(\Cl(X), \BK^{\times})$ on $R(X)$ which is given by the formula:
$$g\circ x_{\rho} = g([D_{\rho}])x_{\rho}.$$
This action defines an action of $G_X$ on $\BA^d$ where $d = |\Sigma(1)|.$ The subset
 $$Z = \{(x_{\rho}) \in \BA^d \mid \prod_{\rho \notin \sigma(1)} x_\rho = 0,\ \forall \sigma \in \Sigma \}$$
is $G_X$-invariant. Denote by $\widehat{X}$ the set $\BA^d\setminus Z$. It was proven in \cite{COX2} that there is a regular map $\pi: \widehat{X} \to X$ which is a categorical quotient with respect to the $G_X$-action. Moreover, when $X$ is simplicial then $\pi$ is a geometric quotient. An action of a unipotent group $U$ on $X$ can be lifted to an action of the group $G_X\times U$ on $\BA^d$ such that $\widehat{X}$ is $G_X\times U$-invariant. Therefore, there are finitely many $U$-orbits on $X$ if and only if there are finitely many $G_X\times U$-orbits on $\widehat{X}.$

We choose some specific partial order $\prec$ on rays of $\Sigma$ and define a specific maximal unipotent group $U$ in $\Aut(X)$ with respect to this partial order; see Section \ref{mainsec}. Next, we give the definition of a \emph{basic} subset in $\Sigma(1)$. Namely, for each subset $A\subseteq \Sigma(1)$ one can consider the submonoid $\Gamma(A)$ in $\Cl(X)$ which is generated by the set $\{[D_{\rho}] \mid \rho \in A\}$ as monoid. We say that a subset $A\subseteq \Sigma(1)$ is \emph{basic} if $[D_{\rho}] \notin \Gamma(A\setminus\{\rho\})$ for all $\rho \in A.$

For each basic subset $A$ we define a subset $\widehat{A}$ consisting of all $\rho' \in \Sigma(1)$ such that $[D_{\rho'}]\notin \Gamma(A)$ or $[D_{\rho'}] = [D_{\rho}]$ for some $\rho \in A$ but $\rho' \prec \rho.$ Then for any basic subset $A \subseteq \Sigma(1)$ we define a subset 
$$Z_A = \{(x_{\rho_1},\ldots, x_{\rho_d}) \in \mathbb{A}^d \mid x_{\rho} \neq 0\ \text{for all} \ \rho \in A \ \text{and} \ x_{\rho'} =  0 \ \text{for all}\ \rho' \in \widehat{A} \} \subseteq \BA^d.$$

We show that each subset $Z_A$ is either a single $G_X\times U$-orbit or a union of infinitely many $G_X\times U$-orbits. The first case holds if and only if the set of irreducible elements in $\Gamma(A)$ is $\mathbb{Q}$-linearly independent. Then we show that $\BA^d$ is a disjoint union of the subsets $Z_A$ and give a criterion when $Z_A \subseteq \widehat{X}.$ Finally, we prove that all subsets $Z_A$ in $\widehat{X}$ are single $G_X\times U$-orbits if and only if all monoids $\Gamma(\sigma)$ are free. 

The structure of the text is as follows. In Section \ref{prelim} we recall the basic facts on toric varieties, automorphism group of complete toric varieties and radiant toric varieties. In Section \ref{mainsec} we prove the main result. Finally, in Section \ref{exampsec} we study the case of surfaces and toric varieties with $\Cl(X) \simeq \BZ$. We also describe the orbits of a maximal unipotent group on these varieties.

\subsection*{Data availability statement and conflict of interest statement}
The author declares that the data supporting the findings of this study are available within the paper. Also the author has no conflicts of interest to declare.

\section{Toric varieties}\label{prelim}
\subsection{Basic facts}
Let $T = (\BK^{\times})^n$ be an algebraic torus. A normal algebraic variety $X$ is called \emph{toric}, if there is a faithful action of $T$ on $X$ with an open orbit. Let $N \simeq \BZ^n$ be the lattice of one-parameter subgroups of $T$ and $M = \mathrm{Hom}(N, \BZ) \simeq \BZ^n$ be the dual lattice of characters of $T$. By $N_{\BQ}$ and $M_{\BQ}$ we mean $\BQ$-vector spaces $N\otimes_{\BZ} \BQ$ and $M\otimes_{\BZ}\BQ$ respectively. There is a natural pairing 
$$\langle \cdot, \cdot \rangle,\ M\times N \to \BZ,\ \langle m, n \rangle = m(n),\ \ m\in M,\ n\in N $$
which can be extended to a bilinear function $\langle \cdot, \cdot \rangle: M_\BQ\times N_\BQ \to \BQ.$

By \emph{cone} in $N_{\BQ}$ (or $M_{\BQ}$) we mean a convex polyhedral cone. We say that a cone is \emph{strongly convex} if it contains no non-zero linear subspaces. We recall that a finite set of strongly convex cones $\Sigma$ in $N_\BQ$ is called a \emph{fan} if the following conditions are satisfied:
\begin{enumerate}
    \item if $\sigma \in \Sigma$ then $\Sigma$ contains all faces of $\sigma$;
    \item if $\sigma_1, \sigma_2 \in \Sigma$ then the intersection $\sigma_1\cap \sigma_2$ is a common face of $\sigma_1$ and $\sigma_2$.
\end{enumerate}

It is well known that every toric variety $X$  has a corresponding fan $\Sigma$ in $N_\BQ$, and $X$ is uniquely determined by $\Sigma$. One can find basic facts on toric varieties in books \cite{COX} and \cite{FULTON}. It also well known that a toric variety $X$ is complete if and only if the corresponding fan is complete, that is
$$\bigcup_{\sigma \in \Sigma} \sigma = N_\BQ.$$

If $\sigma$ is a cone, then by $\sigma(1)$ we denote the set of rays of $\sigma.$ If $\Sigma$ is a fan, then by $\Sigma(1)$ we denote the set of rays in $\Sigma.$  

Each ray $\rho \in \Sigma(1)$ defines a prime $T$-invariant Weil divisor $D_\rho$ on the corresponding toric variety $X$. The divisors $D_\rho$  form a basis of the free abelian group $\mathrm{Div}_T(X)$~--- the group of $T$-invariant Weil divisors. Moreover, the divisor class group $\Cl(X)$ of $X$ is generated by the classes $[D_\rho].$

For each $\rho \in \Sigma(1)$ we denote by $n_\rho \in N$ the primitive vector on $\rho$. Then there is an exact sequence
\begin{equation}\label{ex1}
   0 \longrightarrow M \longrightarrow \mathrm{Div}_T(X) \longrightarrow \Cl(X) \longrightarrow 0, 
\end{equation}
where the second arrow is the map
$$m \to \sum_{\rho \in \Sigma(1)} \langle m, n_\rho \rangle D_\rho$$
and the third arrow is the map $D \to [D] \in \Cl(X).$

\subsection{Automorphism group.} In \cite{COX2} the automorphism group of complete toric varieties was described. We recall this description. Let $X$ be a complete toric variety and $\Sigma$ the corresponding complete fan. The \emph{Cox ring} of $X$ is the following algebra of polynomials

$$R(X) = \BK[x_{\rho} \mid \rho \in \Sigma(1)].$$
 Clearly, the Cox ring is the ring of regular functions on affine space $\BA^d$, where $d = |\Sigma(1)|$. The affine space $\BA^d$ is called \emph{total coordinate space}.

 There is $\Cl(X)$-grading on $R(X)$ where $\mathrm{deg} (x_{\rho}) = [D_\rho].$ We consider the group $G_X = \mathrm{Hom}(\Cl(X), \BK^{\times})$ which is a direct product of a torus and a finite abelian group. Then the $\Cl(X)$-grading on $R(X)$ defines an action of $G_X$ on $R(X):$

 $$g\circ x_{\rho} = g([D_{\rho}])x_{\rho},$$
 where $g \in G_X.$ There is a subset $Z \subseteq \BA^d$ which is given by the equations
 $$\prod_{\rho \notin \sigma(1)} x_\rho = 0,\ \sigma \in \Sigma.$$
It is $G_X$-invariant subset. We denote by $\widehat{X}$ the complement $\BA^d\setminus Z.$ The set $\widehat{X}$ is covered by affine open $G_X$-invariant subsets 
$$V_{\sigma} = \{(x_{\rho_1},\ldots, x_{\rho_d)} \in \mathbb{A}^d \mid \prod_{\rho \notin \sigma(1)}x_{\rho} \neq 0\},\ \sigma \in \Sigma.$$ 
At the same time, each cone $\sigma \in \Sigma$ corresponds to an affine $T$-invariant open subset $X_{\sigma} \subset X.$ It turns out that $X_{\sigma} \simeq \Spec[V_{\sigma}]^{G_X}$ and there is a categorical quotient $\pi_{\sigma}: V_{\sigma} \to X_{\sigma}$. The morphisms $\pi_{\sigma}$ and $\pi_{\sigma'}$ for $\sigma, \sigma' \in \Sigma$ coincide on the intersection $V_{\sigma}\cap V_{\sigma'}.$ So there is a morphism $\pi:\widehat{X} \to X$ which coincide with $\pi_{\sigma}$ on $V_{\sigma}$. The morphism $\pi$ is a categorical quotient so $\widehat{X}$ is a categorical quotient space $\WX/\!/G_X$.

We recall that a cone $\sigma$ is called \emph{simplicial} if primitive vectors on rays $\rho \in \sigma$ are linearly independent. The fan is simplicial if all cones in $\Sigma$ are simplicial. We say that a toric variety $X$ is simplicial if the corresponding fan is simplicial. When $X$ is simplicial then the quotient $\pi: \widehat{X} \to X$ is a geometric quotient.

Denote by $\mathrm{Aut}_g(R(X))$ the group of homogeneous automorphisms of $R(X)$ with respect to the $\Cl(X)$-grading. Here, by homogeneous automorphism we mean an automorphism that take homogeneous polynomials into homogeneous ones, possibly of a different degree. Then $G_X$ is a normal subgroup in $\mathrm{Aut}_g(R(X))$ and $\mathrm{Aut}_g(R(X))$ preserves $\WX.$ Therefore, each automorphism in $\mathrm{Aut}_g(R(X))$ defines an automorphism of $X$. Moreover, there is an exact sequence

$$1 \longrightarrow G_X \longrightarrow \mathrm{Aut}_g(R(X)) \longrightarrow \Aut(X) \longrightarrow 1,$$
where the second arrow is the embedding of $G_X$ in $\mathrm{Aut}_g(R(X))$ and the third arrow was described above.

One can describe the set of generators in $\Aut_g(R(X)).$ Firstly, there is an action of the torus $(\BK^{\times})^d$ on $R(X)$. Let $\Sigma(1) = \{\rho_1,\ldots, \rho_d\}$. Then the action of $(\BK^{\times})^d$ on $R(X)$ is given by the following rule:
$$(t_{\rho_1},\ldots, t_{\rho_d})\circ x_{\rho} = t_{\rho}x_{\rho},$$
where $(t_{\rho_1},\ldots, t_{\rho_d}) \in (\BK^{\times})^d.$  This torus contains $G_X$ and acts on $R(X)$ by homogeneous automorphisms. Therefore, it is embedded in $\Aut_g(R(X))$. The image of $(\BK^{\times})^d$ in $\Aut(X)$ is the torus $T$.

Further, for each $\rho \in \Sigma(1)$ we consider the set $\FR_{\rho} \subseteq M$ of vectors $e\in M$ such that
$$\langle e, n_{\rho}\rangle = -1, \ \text{and}\ \langle e, n_{\rho'}\rangle \geq 0 \ \text{for all} \ \rho'\neq \rho,\ \rho' \in \Sigma(1).$$
The elements of the set $\FR = \sqcup_{\rho} \FR_{\rho}$ are called \emph{Demazure roots} of $\Sigma.$ The roots in $\mathfrak{S} = \FR \cap (-\FR)$ are called \emph{semisimple} and all other roots  
 are called \emph{unipotent}. We denote the set of unipotent roots by $\mathfrak{U}.$

 Denote by $\BG_a$ the additive group of the field $\BK$. Then each Demazure root $e\in \FR_\rho$ defines a $\BG_a$-action on $R(X)$. One can consider the monomial 
 $$X_e = \prod_{\rho'\neq \rho}x^{\langle e, \rho'\rangle}_{\rho'}.$$
 Then an element $s\in \BG_a$ acts as follows:
 
 $$s\circ x_{\rho} = x_{\rho} + sX_e, \ s\circ x_{\rho'} = x_{\rho'},\ \forall \rho'\neq \rho.$$

 This action of $\BG_a$ on $R(X)$ defines a one-dimensional unipotent subgroup in $\Aut_g(R(X)).$ We denote this subgroup by $U_e$. The image of $U_e$ in $\Aut(X)$ defines a $\BG_a$-action on $X$ which is normalized by $T$.

 Finally, one can consider the group $\Aut(N, \Sigma)$ which consists of automorphisms of the lattice $N$ preserving $\Sigma.$ Then for each $\psi \in \Aut(N, \Sigma)$ we define the following automorphism of $R(X)$:
 $$\overline{\psi}(x_{\rho}) = x_{\psi(\rho)}, \ \forall \rho \in \Sigma(1).$$

 The map $\psi \to \overline{\psi}$ defines the embedding of $\Aut(N, \Sigma)$ into $\Aut_g(R(X)).$ Then $\Aut_g(R(X))$ is a linear algebraic group which is generated by the subgroups  $(\BK^{\times})^d$,  $U_e$ ($e\in \FR$) and $\Aut(N, \Sigma).$ Moreover, the subgroups $(\BK^{\times})^d$ and $U_e$ generate the connected component of identity of $\Aut_g(R(X))$. The subgroups $(\BK^{\times})^d$ and $U_e$ with $e\in \mathfrak{S}$ generate a maximal reductive subgroup and the subgroups $U_e$ with $e\in \mathfrak{U}$ generate the unipotent radical of $\Aut_g(R(X)).$ 

Now let us choose a vector $v\in N$ such that $\langle e, v\rangle \neq 0$ for all $e \in \mathfrak{S}.$ Then we can define the subset 
$$\FS^+ = \{e \in \FS \mid \langle e, v \rangle > 0\}.$$
Then the subgroups $U_e$ with $e\in \FS^+ \sqcup \mathfrak{U}$ generate a maximal unipotent subgroup in $\Aut_g(R(X))$ which is normalized by the torus $(\BK^{\times})^d.$

\subsection{Radiant toric varieties.} \label{rad} In \cite{AR} complete toric varieties were described that admit the action of a unipotent group with an open orbit. Here we recall this result. We say that a complete toric variety is \emph{radiant} if a maximal unipotent subgroup in $\Aut(X)$ acts on $X$ with an open orbit. We say that a fan $\Sigma$ in $N_\BQ$ is \emph{bilateral} if there is a basis $e_1, \ldots, e_n \in N$ such that the vectors $e_i$ generate rays $\varepsilon_i$ in $\Sigma(1)$ and all other rays $\tau_1,\ldots, \tau_k$  in $\Sigma(1)$ lie in the negative orthant with respect to this basis. We will call rays $\varepsilon_1, \ldots, \varepsilon_n$ \emph{positive} and rays $\tau_1,\ldots, \tau_k$ \emph{negative}. An example of a bilateral fan can be found in Figure \ref{Fig1}. 
\begin{figure}[h]
\begin{center}
\input{picture}
\end{center}
\caption{Example of bilateral fan}\label{Fig1}

\end{figure}

\begin{theorem}\cite[Theorem 3]{AR}
Let $X$ be a complete toric variety and $\Sigma$ is the corresponding fan. The following conditions are equivalent:

\begin{enumerate}
    \item the variety $X$ is radiant;
    \item the fan $\Sigma$ is bilateral.
\end{enumerate}
 
\end{theorem}

When $X$ is a radiant complete toric variety then the divisor class group $\Cl(X)$ is a free finitely generated abelian group with the basis $[D_{\tau_1}],\ldots, [D_{\tau_k}];$ see \cite[Corollary 3]{Ds2}.  
 
The structure of maximal unipotent subgroups in $\Aut(X)$ of a radiant complete toric variety $X$ was studied in \cite{APS}.

\section{Simplicial toric varieties admitting an action of a unipotent group with a finite number of orbits}\label{mainsec}

In this section we describe complete simplicial toric varieties on which a maximal unipotent group acts with a finite number of orbits.

Let $X$ be a complete radiant toric variety and $\Sigma$ corresponding bilateral complete fan. Let $\varepsilon_1,\ldots, \varepsilon_n$ be the positive rays in $\Sigma(1)$ and $\tau_1,\ldots, \tau_k$ be the negative rays. Then the divisor class group $\Cl(X)$ is a free finitely generated abelian group with basis $[D_{\tau_1}],\ldots, [D_{\tau_k}];$ see \cite[Corollary 3]{Ds2}. 

If we denote $e_1,\ldots, e_n \in N$ the primitive vectors on $\varepsilon_1,\ldots, \varepsilon_n$ and by $n_{1}, \ldots, n_{k} \in N$ the primitive vectors on the rays $\tau_1,\ldots, \tau_k$  then $$[D_{\varepsilon_i}] = -\sum_{j}\langle e^i, n_{j}\rangle[D_{\tau_j}],$$
where $e^1,\ldots, e^n \in M$ is the dual basis to $e_1,\ldots, e_n$. Consider a submonoid $\Cl^+(X)$ in $\Cl(X)$ generated by $[D_{\tau_1}],\ldots, [D_{\tau_k}].$ Then $\Cl^+(X) \simeq (\BZ_{\geq 0})^k$ and $[D_{\varepsilon_i}] \in \Cl^+(X)$ for all $i$. 

\begin{example}
    Let $X$ be a complete toric variety corresponding to the bilateral fan shown in Figure \ref{Fig2}
\begin{figure}[h]
\begin{center}
\input{Example2Fan}
\end{center}
\caption{The fan of $X$.}\label{Fig2}
\end{figure}
\begin{figure}[h]
\begin{center}
\input{Example2Cl}
\end{center}
\caption{Classes of prime T-invariant divisors.}\label{Fig3}
\end{figure}

Then the divisors $[D_{\tau_1}]$ and $[D_{\tau_2}]$ form a basis of $\Cl(X) \simeq \BZ^2$ and $[D_{\varepsilon_1}] = [D_{\tau_2}],\ [D_{\varepsilon_2}] = 2[D_{\tau_1}] + [D_{\tau_2}];$ see Figure \ref{Fig3}.

\end{example}

Let $R$ be the Cox ring of $X$ and $G_X = \mathrm{Hom}(\Cl(X), \BK^*)$. We choose a vector $v\in N$ such that
$$0> \langle e^1, v\rangle > \langle e^{2}, v\rangle  > \ldots > \langle e^n, v \rangle . $$
Denote by $U$ the corresponding maximal unipotent subgroup in $\Aut_g(R)$; see Section~\ref{prelim}. 

For positive rays we write $\varepsilon_i < \varepsilon_j$ if and only if $i<j$. Then we can define the following partial order on $\Sigma(1)$. We write that $ \rho \prec \rho'$ for $\rho,\rho' \in \Sigma(1)$ if and only if one of the following holds:

\begin{enumerate}
    
    \item $[D_{\rho}] = [D_{\rho'}]$ and $\rho, \rho'$ are positive rays with $\rho < \rho';$  
    \item $[D_{\rho}] = [D_{\rho'}]$ and $\rho$ is negative and $\rho'$ is positive.
\end{enumerate}

\begin{lemma}\label{lem1}
\begin{enumerate}
    \item We have $\varepsilon_i \prec \varepsilon_j$ if and only if the element  $e = e^i - e^j \in M$ is a Demazure root in $\FR_{\varepsilon_j}$ with $U_e \subseteq U$ and $\langle e, n_{l}\rangle = 0$  for all $l = 1,\ldots, k$;
    \item We have $\tau_i \prec \varepsilon_j $ if and only if   the element $e =  -e^j$ is a Demazure root in $\FR_{\varepsilon_j}$ with $U_e \subseteq U$ such that $\langle e, n_{l}\rangle = 0$ for all  $l \neq i$ and $\langle e, n_i \rangle = 1.$ 
    \item We have $\rho \prec \rho'$ for some $\rho,\rho' \in \Sigma(1)$ if and only if there is a unique semisimple Demazure root $e \in \FR_{\rho'}$ with $U_e \subseteq U$ such that 
    \begin{equation}\label{eqlem}
    \langle e, n_\rho \rangle = 1,\ \langle e, n_{\rho'} \rangle = -1\  \text{and}\ \langle e, n_{\rho''}\rangle\ = 0\ \text{for all}\ \rho'' \neq \rho, \rho'.\end{equation}
\end{enumerate}    
\begin{proof}
    (1) Suppose $\varepsilon_i \prec \varepsilon_j$. Then $[D_{\varepsilon_i}] = [D_{\varepsilon_j}]$ and $i<j$. But at the same time
    $$[D_{\varepsilon_i}] - [D_{\varepsilon_j}]= -\sum_{l}\langle e^i, n_{l}\rangle[D_{\tau_l}] + \sum_{l}\langle e^j, n_{l}\rangle[D_{\tau_l}] = -\sum_{l}\langle e^i - e^j, n_{l}\rangle[D_{\tau_l}] = 0.$$
    So $\langle e^i - e^j, n_l\rangle = 0$ for all $l = 1,\ldots, k.$ 
    
   We have $\langle e^i - e^j, e_l \rangle \geq 0$ when $l \neq j$ and $\langle e^i - e^j, e_j \rangle = -1$. Therefore, $e = e^i - e^j$ is a semisimple Demazure root. At the same time
    $$\langle e^i - e^j, v \rangle = \langle e^i, v \rangle - \langle e^j, v \rangle > 0$$
    since $i<j$. So $U_e \subseteq U.$ 

    Conversely, if $e = e^i - e^j$ is a Demazure root with $U_e\subseteq U$ and $\langle e, n_l \rangle = 0$ for all $l = 1,\ldots, k$ then 
    $$\sum_{\rho \in \Sigma(1)}\langle e, n_{\rho}\rangle D_{\rho} = D_{\varepsilon_i} - D_{\varepsilon_j}.$$
    From exact sequence \ref{eq1} we obtain that $[D_{\varepsilon_i}] = [D_{\varepsilon_j}].$ Since $U_e\subseteq U$ we have 
    $$\langle e, v \rangle = \langle e^i, v \rangle - \langle e^j, v \rangle > 0.$$
    It implies $i<j.$

    (2) Now suppose that $\tau_i \prec \varepsilon_j$. Then $[D_{\tau_i}] = [D_{\varepsilon_j}]$ so 
    
    $$[D_{\tau_i}] - [D_{\varepsilon_j}]= [D_{\tau_i}] + \sum_{l}\langle e^j, n_{l}\rangle[D_{\tau_l}] = 0.$$
    It implies that $\langle -e^j, n_l \rangle = 0$ when $l \neq i$ and $\langle -e^j, n_i\rangle = 1.$ Hence, $e = -e^j$ is a Demazure root. Since $\langle -e^j, v\rangle > 0 $ we have $U_e \subseteq U.$ 
    
    Conversely, suppose that $e = -e^j$ is a Demazure root with $U_e \subseteq U$ such that $\langle e, n_l\rangle = 0$ when $l\neq i$ and $\langle e, n_i \rangle = 1.$ Then
    $$[D_{\varepsilon_j}] = -\sum_{j}\langle e^i, n_{j}\rangle[D_{\tau_j}] = [D_{\tau_i}].$$

    (3) Since the fan $\Sigma$ is complete equations (\ref{eqlem}) define an element in $N$ uniquely. It is easy to check that if $e$ is a Demazure root from points (1) or (2), then $-e$ is also a Demazure root. So this part follows from parts (1) and (2).

\end{proof}

\end{lemma}

For a subset $A \subseteq \Sigma(1)$ we can consider the submonoid $\Gamma(A)$ in $\Cl^+(X)$ generated by the set $\{[D_{\rho}] \mid \rho \in A\}.$ We say that a subset $A \subseteq \Sigma(1)$ is \emph{basic} if $[D_{\rho}] \notin \Gamma(A\setminus\{\rho\})$ for all $\rho \in A.$ 

For a basic subset $A$ we can consider the subset $\widehat{A}$ consisting of all $\rho' \in \Sigma(1)$ that satisfy one of the following two  conditions:
\begin{enumerate}
    \item $[D_{\rho'}] \notin \Gamma(A)$ or
    \item $[D_{\rho'}] = [D_{\rho}]$ for some $\rho \in A$ but $\rho' \prec \rho.$
\end{enumerate}
It is clear that the sets $A$ and $\widehat{A}$ do not intersect.

Now for any basic subset $A \subseteq \Sigma(1)$ we can define a subset $Z_A$ of the total coordinate space $\BA^d = \BA^{n+k}$ as follows:
$$Z_A = \{(x_{\rho_1},\ldots, x_{\rho_d}) \in \mathbb{A}^d \mid x_{\rho} \neq 0\ \text{for all} \ \rho \in A \ \text{and} \ x_{\rho'} =  0 \ \text{for all}\ \rho' \in \widehat{A} \}.$$

\begin{lemma}\label{lem2} Let $A$ be a basic subset in $\Sigma(1)$ and $e \in \FR_{\rho_0}$ a Demazure root for some $\rho_0 \in \Sigma(1)$ such that $U_e \subseteq U.$
\begin{enumerate}
    \item If $\rho_0 \in A\sqcup \widehat{A}$ then $U_e$ acts trivially on $Z_A.$
    \item The subset $Z_A$ is $G_X\times U$-invariant.
    \item  For a point $p\in Z_A$ the coordinates $x_{\rho}$ with $\rho \in A$ are constant  with respect to the action of $U$.
\end{enumerate}
\end{lemma}

\begin{proof}

Let us prove the part (1).

\textbf{Case 1:}  $\rho_0 \in A$. From exact sequence \ref{ex1} we have 

\begin{equation}\label{eq1}
\sum_{\rho \in \Sigma(1)}\langle e, n_{\rho} \rangle [D_{\rho}] = \sum_{\rho \neq \rho_0}\langle e, n_{\rho}\rangle[D_{\rho}] - [D_{\rho_0}] = 0.
\end{equation}

\textbf{Case 1.a:} Suppose that there is $\rho'$ with $[D_{\rho'}] \notin \Gamma(A)$ and $\langle e, n_{\rho'}\rangle > 0.$ Then $\rho' \in \widehat{A}$ and $x_{\rho'} = 0$ on $Z_A.$ But then the monomial 
$$X_e = \prod_{\rho \neq \rho_0} x_{\rho}^{\langle e, n_{\rho}\rangle}$$
is divisible by $x_{\rho'}$, and therefore is equal to zero on $Z_A$. Then $U_e$ acts trivially on $Z_A.$

\textbf{Case 1.b:} Suppose that for all $\rho'$ with $[D_{\rho'}] \notin \Gamma(A)$ we have $\langle e, n_{\rho'}\rangle = 0.$ We consider the set $C = \{\rho \in \Sigma(1) \mid \langle e, n_{\rho} \rangle >0 \}.$ Then for each $\rho \in C$ we have $[D_{\rho}] \in \Gamma(A)$. So for all $\rho \in C$ there are numbers $b^{\widehat{\rho}}_{\rho} \in \BZ_{\geq 0}$ such that
$$[D_{\rho}] = \sum_{\widehat{\rho} \in A}b^{\widehat{\rho}}_{\rho}[D_{\widehat{\rho}}]$$
and
 \begin{equation}\label{eq123}[D_{\rho_0}] = \sum_{\rho \in C} \sum_{\widehat{\rho} \in A}\langle e, n_{\rho}\rangle b^{\widehat{\rho}}_{\rho}[D_{\widehat{\rho}}] = \sum_{\rho \in C} \sum_{\widehat{\rho} \in A}c^{\widehat{\rho}}_{\rho}[D_{\widehat{\rho}}],\end{equation}
 where $c^{\widehat{\rho}}_{\rho} = \langle e, n_{\rho}\rangle b_{\rho}^{\widehat{\rho}}.$ If $c^{\rho_0}_{\rho} = 0$ for all $\rho \in C$ then $[D_{\rho_0}] \in \Gamma(A\setminus\{\rho_0\}).$ But $A$ is basic. So there is $\rho$ with $c^{\rho_0}_{\rho} > 0$. But then $b^{\rho_0}_{\rho} = 1$ and all other coefficients in equation (\ref{eq123}) are zero. Therefore, equation (\ref{eq1}) has the form
$$[D_{\rho}] - [D_{\rho_0}] = 0.$$
It implies that $e$ is a Demazure root from Lemma \ref{lem1} and $\rho \prec \rho_0.$ But then $\rho \in \widehat{A}$. Since $X_e$ divisible by $x_{\rho}$ the group $U_e$ acts trivially on $Z_A$.

\textbf{Case 2:}  $\rho_0 \in \widehat{A}$. Then either  $[D_{\rho_0}] \notin \Gamma(A)$  or $[D_{\rho_0}] = [D_{\rho'}]$ for some $\rho' \in A$ with $\rho_0 \prec \rho'$.

\textbf{Case 2.a:} Suppose $[D_{\rho_0}] \notin \Gamma(A)$. Then equation \ref{eq1} implies that there is $\rho' \in \Sigma(1)$ with $[D_{\rho'}]\notin \Gamma(A) $ and $\langle e, n_{\rho'}\rangle > 0.$ In this case we again obtain that $X_e$ is divisible by $x_{\rho'}$ with $\rho' \in \widehat{A}.$

\textbf{Case 2.b:} Suppose $[D_{\rho_0}] = [D_{\rho'}]$ for some $\rho' \in A$ with $\rho_0 \prec \rho'$. If $\langle e, n_{\rho}\rangle > 0$ for some $\rho$ with $[D_{\rho}] \notin \Gamma(A)$ then the monomial $X_e$ is equal to zero on $Z_A$.  So we can consider only the case when $\langle e, n_{\rho}\rangle > 0$ only when $[D_{\rho}] \in \Gamma(A).$ 

By equation 3 we have 
$$\sum_{\rho \neq \rho_0}\langle e, n_{\rho}\rangle [D_{\rho}] - [D_{\rho_0}] =  0$$
and in the first sum all terms belongs to $\Gamma(A).$ Applying the same arguments from case 1b we obtain that there is only one $\rho \neq \rho_0$ with $\langle e, n_{\rho}\rangle \neq 0$ and $\langle e, n_{\rho}\rangle = 1$. Hence, $e$ is the Demazure root from Lemma \ref{lem1}  and $\rho \prec \rho_0 \prec \rho'.$ We have $[D_{\rho}] = [D_{\rho}] = [D_{\rho'}].$ So $\rho \in \widehat{A}$ and the monomial $X_e$ is zero on $Z_A$.

Now we will prove the parts (2) and (3). It is clear that $Z_A$ is $G_X$-invariant. The group $U$ is generated by the subgroups $U_e$ with $U_e \subseteq U.$ If $e\in \FR_{\rho_0}$ with $\rho_0 \in A\sqcup \widehat{A}$ then $U_e$ acts trivially on $Z_A$ by part (1). If $e\in \FR_{\rho_0}$ with $\rho_0 \notin A\sqcup \widehat{A}$ then the group $U_e$ do not change variables $x_{\rho}$ for $\rho \in A \sqcup \hat{A}$. So $Z_A$ is $U_e$-invariant and for all points $p\in Z_A$ the coordinates $x_{\rho}$ with $\rho \in A$ are constant with respect to the $U$-action.

\end{proof}

Since the group $\Cl(X)$ is a free abelian group of rank $k$ then $\Cl(X)_{\BQ} = \Cl(X)\otimes_{\BZ} \BQ$ is a $\BQ$-vector space of dimension $k$ and $\Cl(X)$ is naturally embedded in $\Cl(X)_{\BQ}.$

\begin{lemma}\label{lem3}
    Let $A$ be a basic subset in $\Sigma(1).$ Then the number of $G_X\times U$-orbits in $Z_A$ is finite if and only if the set $\{[D_{\rho}] \mid \rho \in A\}$ is $\BQ$-linearly independent. In this case the set $Z_A$ is a $G_X\times U$-orbit. 
\end{lemma}
\begin{proof}

    Suppose that the set $\{[D_{\rho} ]\mid \rho \in A \}$ is $\BQ$-linearly dependent. Then there is a nontrivial linear combination $\sum_{\rho}b_{\rho}[D_{\rho}] = 0 $ for some $b_{\rho} \in \BZ.$ Then the function 
    $$f = \prod_{\rho \in A}x_{\rho}^{b_{\rho}}$$
    is $G_X$-invariant. All subgroups $U_e \subseteq U$ such that $e\in \FR_{\rho_0}$ with $\rho_0 \in A\sqcup \widehat{A}$ act trivially on $Z_A$ by Lemma \ref{lem2}. When $\rho_0 \notin A\sqcup \widehat{A}$ the subgroup $U_e$ does not change variables $x_{\rho}$ with $\rho \in A$. Hence, the function $f$ is a regular nonconstant $G_X\times U$-invariant on $Z_A$. Therefore, there are infinitely many orbits of $G_X\times U$ in $Z_A$.

    Now suppose that the set $\{[D_{\rho} ]\mid \rho \in A \}$ is $\BQ$-linearly independent. Since $\Cl(X)$ is a free abelian group of rank $k$, then the group $G_X$ is a torus $(\BK^*)^k$ of dimension $k$ and $\Cl(X)$ is naturally isomorphic to the lattice of characters of $(\BK^*)^k$. One can choose a basis $f_1,\ldots, f_k \in \Cl(X)$ such that 
    $$[D_{\rho}] = c_{\rho}f_{\rho}, \ \forall \rho \in A,$$ 
    where $c_{\rho} \in \BZ$ and $f_{\rho} \in \{f_1,\ldots, f_k\}$. Moreover, $f_{\rho} \neq f_{\rho'}$ when $\rho \neq \rho'\in A$. We denote by $t_1, \ldots, t_k$ the corresponding coordinates on $(\BK^*)^k$. Then 
    $$t\circ x_{\rho} = t_{\rho}^{c_{\rho}}x_{\rho}, \ \text{for} \ \rho\in A.$$
    Then $G_X$-orbit of a point $p\in Z_A$ contains points $q$ whose coordinates $x_{\rho}$ with $\rho \in A$ can be arbitrary nonzero elements of $\BK.$ We will show that by acting with the group $U$ the coordinates $x_{\rho}$ with $\rho \notin A\sqcup \widehat{A}$ can be made arbitrary without changing the coordinates $x_{\rho}$ with $\rho \in A.$ This will show that $Z_A$ is a $G_X\times U$-orbit.

    Let us take $\rho_0 \notin A\sqcup \widehat{A}$. Then $[D_{\rho_0}] \in \Gamma(A)$. Hence,
    $$[D_{\rho_0}] = \sum_{\rho\in A}a_{\rho}[D_{\rho}]$$
    for some $a_{\rho} \in \BZ_{\geq 0}.$ Therefore, 
    $$\sum_{\rho\in A}a_{\rho}[D_{\rho}] - [D_{\rho_0}] = 0.$$
    It follows from exact sequence (\ref{ex1}) that there is a Demazure root $e\in M$ such that $\langle e, n_{\rho_0}\rangle = -1$, $\langle e, n_{\rho}\rangle = a_{\rho}$ for $\rho \in A$ and $\langle e, n_{\rho'}\rangle = 0$ for $\rho' \notin A\sqcup \{\rho_0 \}.$ Therefore, $e \in \FR_{\rho_0}$. If $\sum_{\rho \in A}a_{\rho}>1$ then $e$ is a unipotent root and $U_e \subseteq U$. 

    If $\sum_{\rho \in A}a_{\rho} = 1$ then $[D_{\rho_0}] = [D_{\rho}]$ for some $\rho \in A$ with $\rho \prec \rho_0$. Then $e$ is a Demazure root from Lemma \ref{lem1} and $U_e\subseteq U.$

    The monomial $X_e = \prod_{\rho \in A}x_{\rho}^{\langle e, n_{\rho} \rangle}$ is not equal to zero on $Z_A.$ So $U_e$-orbit of a point $p \in Z_A$ is the set of points $q$ such that $x_{\rho'}(q) = x_{\rho'}(p)$ for all $\rho' \neq \rho_0$ and $x_{\rho_0}(q)$ can be arbitrary. Therefore, by acting with the group $U$ the coordinates $x_{\rho}$ with $\rho \notin A\sqcup \widehat{A}$ can be made arbitrary without changing the coordinates $x_{\rho}$ with $\rho \in A.$

\end{proof}

\begin{remark}
    The set $\{[D_{\rho}] \mid \rho \in \Sigma(1)\}\subseteq \Cl(X)_{\BQ} $ is Gale dual to the set $\{n_{\rho} \mid \rho \in \Sigma(1)\}$. So Lemma \ref{lem3} can be formulated in the following way. The number of $G_X\times U$-orbits in $Z_A$ is finite if and only if the set $\{n_{\rho} \mid \rho \notin A\}$ generates $N_\BQ.$
\end{remark}

It remains to understand which sets $Z_A$ are contained in $\widehat{X}$.

\begin{lemma}\label{lem4}
    Let $\mathcal{A}$ be the set of all basic subsets in $\Sigma(1)$. Then
    \begin{enumerate}
        \item $\BA^{d} = \sqcup_{A\in \mathcal{A}} Z_A;$
        \item the set $Z_{A}$ is contained in $\widehat{X}$ if and only if there is a cone $\sigma \in \Sigma$ such that $\widehat{A} \subseteq \sigma(1).$
    \end{enumerate}
\end{lemma}
\begin{proof}
    (1) Let $A \in \mathcal{A}$ and $p \in Z_A.$ As we saw in the proof of Lemma \ref{lem3}, when $U$ acts on $Z_A$, only the coordinates $x_{\rho}$ with $\rho \in A \sqcup \widehat{A}$ remain constant.Then the set $A$ is the set of coordinates $x_{\rho}$ which are nonzero at the point $p$ and constant on $U$-orbit of $p$. Then $Z_A \cap Z_B = \emptyset$ if $A\neq B.$ 

    For a point $p\in \BA^d$ consider the set $A_p \subseteq \Sigma(1)$ of all rays $\rho \in \Sigma(1)$ such that $x_{\rho}$ are nonzero at the point $p$ and constant on $U$-orbit of $p$. Let us prove that $A_p$ is basic. 
    
    Let $\rho$ be an element of $A_p$. Assume that $[D_{\rho}] = [D_{\rho'}]$ for  some $ \rho' \in A_p$ with $\rho \prec \rho'$. Then we have the Demazure root $e \in \FR_{\rho'}$ from Lemma \ref{lem1} with $U_e \subseteq U$. The subgroup $U_e$ changes $x_{\rho'}$ if $x_{\rho} \neq 0$. This contradicts the definition of the set $A_p$.

    Now suppose that
    $$[D_{\rho}] = \sum_{\rho' \in A_p\setminus \{\rho\}}a_{\rho'}[D_{\rho'}]$$
    for some integers $a_{\rho'} \in \BZ_{\geq 0}$ with $\sum_{\rho'}a_{\rho'}>1$. Then there is a Deamzure root $e\in \FR_{\rho}$ such that $\langle e, n_{\rho'}\rangle = a_{\rho'}$ for $\rho' \in A_p\setminus \{\rho\}$ and $\langle e, n_{\rho'}\rangle  = 0$ for $\rho' \notin A_p$. Then $e$ is unipotent. Hence, $U_e\subseteq U$ and $U_e$ changes $x_{\rho}$ on $U$-orbit of $p$. Again, it contradicts the definition of the set $A_p$. Therefore, $[D_{\rho}] \notin \Gamma(A_p\setminus \{\rho\})$ and $A_p$ is a basic subset. Obviously $p \in Z_{A_p}.$

    (2) Let $A$ be a basic subset. For any point $p \in Z_A$ there is a point $q$ in $U$-orbit of $p$ such that $x_{\rho}(q) = 0$ if and only if $\rho \in \widehat{A}.$ Then $q \in \widehat{X}$ if and only if there is a cone $\sigma \in \Sigma$ with $\rho \in \sigma(1)$ for all $\rho \in \widehat{A}$. Since $Z_A$ and $\widehat{X}$ are $U$-invariant, we obtain that $Z_A\subseteq \widehat{X}$ if and only if there is a cone $\sigma \in \Sigma$ with $\rho \in \sigma(1)$ for all $\rho \in \widehat{A}$.
\end{proof}

Thus, we obtain the following criterion.

\begin{proposition}\label{mprop}
    Let $X$ be a complete simplicial toric variety and $\Sigma$ be the corresponding fan. Then a maximal unipotent subgroup in $\Aut(X)$ acts on $X$ with a finite number of orbits if and only if $X$ is radiant and for all basic subsets $A \subseteq \Sigma(1)$ such that $\widehat{A} \subseteq \sigma(1)$ for some $\sigma \in \Sigma$ the set $\{[D_{\rho}]\mid \rho \in A\}$ is $\BQ$-linear independent.
\end{proposition}
\begin{proof}
    Suppose that a maximal unipotent subgroup acts on $X$ with a finite number of orbits. Then $X$ is radiant. Consider the maximal unipotent subgroup $U$ in $\Aut_g(R(X))$ as above. Denote by $\overline{U}$ its image in $\Aut(X)$. Then $\overline{U}$ is a maximal unipotent subgroup in $\Aut(X)$. Since all maximal unipotent subgroups in $\Aut(X)$ are conjugated, the group $\overline{U}$ acts on $X$ with a finite number of orbits.

    The variety $X$ is simplicial. So $X$ is a geometric quotient space  $\widehat{X}/G_X$. Then, there is a finite number of $\overline{U}$-orbits on $X$ if and only if there is a finite number of $G_X\times U$-orbits on $\widehat{X}$. Now the statement follows from Lemmas \ref{lem3} and \ref{lem4}. The converse also follows from Lemmas \ref{lem3} and \ref{lem4}.
\end{proof}

We will call a basic subset $A\subseteq \Sigma(1)$ \emph{minimal} if the condition $[D_{\rho}] = [D_{\rho'}]$ for some $\rho \in A, \rho' \in \Sigma(1)$ implies $\rho \prec \rho'$ or $\rho = \rho'.$

\begin{lemma}
    
    Let $A\subseteq \Sigma(1)$ be a basic subset.
    
    \begin{enumerate}
        \item There is a unique minimal basic subset $A^0$ with $\Gamma(A) = \Gamma(A^0).$ Moreover, the sets  $\{[D_{\rho}]\mid \rho \in A\}$ and $\{[D_{\rho}]\mid \rho \in A^0\}$ coincide.
        \item The set $Z_A$ is a single $G_X\times U$-orbit if and only if $Z_{A^0}$ is.
        \item $\widehat{A^0} = \{\rho \in \Sigma(1) \mid [D_{\rho}] \notin \Gamma(A)\} \subseteq \widehat{A}.$
        \item If $Z_A \subseteq \widehat{X}$ then $Z_{A^0} \subseteq \widehat{X}.$
    \end{enumerate}

\end{lemma}
\begin{proof}
    
(1) The set $\{[D_{\rho}]\mid \rho \in A\}$ it the set of all irreducible elements in the semigroup~$\Gamma(A).$ The set $A^0$ is uniquely determined by this set. 

(2) Follows from Lemma \ref{lem3}.

(3) Trivially.

(4) Follows from point (2) of Lemma \ref{lem4}.

\end{proof}
For a cone $\sigma \in \Sigma$ we denote by $\Gamma(\sigma)$ the submonoid in $\Cl(X)$ generated by the set $\{[D_{\rho}\mid \rho \notin \sigma \}$.  We say that an abelian finitely generated monoid $\Gamma$ is \emph{free} if it is isomorphic to the monoid $(\BZ_{\geq 0})^s$ for some $s$. If $\Gamma$ is a finitely generated submonoid in $(\BZ_{\geq 0})^d$ then $\Gamma$ is free if and only if the set of irreducible elements in $\Gamma$ is finite and $\BQ$-linear independent.

Now we are ready to prove the main result. 

\begin{proof} (Proof of Theorem \ref{mainth}.)

    Suppose a maximal unipotent subgroup acts on $X$ with a finite number of orbits. Then $X$ is radiant. Consider a cone $\sigma \in \Sigma. $ There is a minimal basic subset $A$ such that the set $\{ [D_{\rho}] \mid \rho \in A\}$ is the set of all irreducible elements in $\Gamma(\sigma).$ Then $\Gamma(\sigma) = \Gamma(A)$ and $\widehat{A} \subseteq \sigma(1).$  By Proposition \ref{mprop} the set $\{[D_{\rho}] \mid \rho \in A\}$ is $\BQ$-linear independent. Then $\Gamma(\sigma)$ is freely generated by $\{[D_{\rho}] \mid \rho \in A\}$.

    Now suppose that $X$ is radiant and all monoids $\Gamma(\sigma)$ are free. Let $A$ be a basic subset with $\widehat{A} \subseteq \sigma(1)$ for some $\sigma \in \Sigma.$ Let us consider the minimal basic subset $A^0$ with $\Gamma(A^0) = \Gamma(A).$ Then $\widehat{A^0} \subseteq \widehat{A} \subseteq \sigma(1).$ Since $\sigma$ is simplicial then there is a cone $\tau\prec \sigma$ with $\tau(1) = \widehat{A^0}. $ Then, $\Gamma(A) =   \Gamma(A^0) = \Gamma(\tau).$ Therefore, $\Gamma(A)$ is free. The set $\{[D_{\rho}] \mid \rho \in A \}$ is the set of irreducible elements in $\Gamma(A)$. Then the set $\{ [D_{\rho}] \mid \rho \in A\}$ is $\BQ$-linear independent. By Proposition \ref{mprop} a maximal unipotent subgroup in $\Aut(X)$ acts on $X$ with a finite number of orbits.

\end{proof}

\section{Examples}\label{exampsec}

Here we give examples of toric varieties on which a unipotent group acts with a finite number of orbits. 

\begin{proposition}\label{projspace}
    Let $X$ be a complete toric variety with $\Cl(X) \simeq \BZ.$ Then a maximal unipotent subgroup in $\Aut(X)$ acts on $X$ with a finite number of orbits if and only if $X$ is a weighted projective space $\BP(1, 1, d_2, \ldots, d_n)$ with $d_i \mid d_{i+1}.$
\end{proposition}

\begin{proof}
    By \cite[Proposition 2]{AR} the variety $X$ is a complete radiant toric variety with $\Cl(X) \simeq \BZ$ if and only if  $X$ is a weighted projective space $\BP(1, d_1, \ldots, d_n)$ where $d_1,\ldots, d_n$ are relatively prime. Let us denote the corresponding fan by $\Sigma$. Then the set $\Sigma(1)$ contains positive rays $\varepsilon_1,\ldots, \varepsilon_n$ and the negative ray $\tau$ which is generated by the vector $n_{\tau} = -(d_1,\ldots, d_n).$ We can assume that $1 \leq d_1\leq d_2 \leq \ldots \leq d_n$. We see that $\Sigma$ is simplicial.

    Then $[D_{\varepsilon_i}] = d_i[D_{\tau}].$ We consider the cone $\sigma_i\in \Sigma$  which is generated by rays $\tau, \varepsilon_1, \ldots, \varepsilon_{i-1}$. Hence, $\Gamma(\sigma_i)$ is generated by the set $[D_{\varepsilon_i}], \ldots, [D_{\varepsilon_n}].$ So $\Gamma(\sigma_i)$ is isomorphic to the submonoid in $\mathbb{Z}_{\geq 0}$ which is generated by the numbers $d_i, \ldots, d_n.$ Therefore, $\Gamma(\sigma_i)$ is free if and only if $d_i$ divides $ d_{i+1},\ldots d_n.$ So if all monoids $\Gamma(\sigma)$ are free then $d_1 \mid d_2 \mid \ldots \mid d_n$. Since $d_1,\ldots, d_n$ are relatively prime we obtain that $d_1 = 1.$

    It is easy to see that if $d_1 = 1$ and $d_2\mid d_3 \mid \ldots \mid d_n$ then all monoids $\Gamma(\sigma)$ are free for all $\sigma \in \Sigma.$
\end{proof}

Now we will consider the two-dimensional case.

\begin{proposition}\label{dim2}
    Let $X$ be a complete toric surface. Then a maximal unipotent subgroup in $\Aut(X)$ acts on $X$ with a finite number of orbits if and only if one of the following holds:
    \begin{enumerate}
        \item $X$ is a weighted projective space $\BP(1, 1, d);$
        \item $X$ is $\BP^1 \times \BP^1;$
        \item $X$ is a Hirzebruch surface $\BF_d.$
    \end{enumerate}
\end{proposition}
\begin{figure}[h]
\begin{center}
\input{Surfaces}
\end{center}
\caption{The fans of $\BP(1, 1, d),\ \BP^1\times \BP^1$ and $\BF_d$.}\label{Fig4}
\end{figure}
\begin{proof}
    Let $X$ be a complete radiant toric surface. Suppose that a maximal unipotent subgroup in $\Aut(X)$ acts on $X$ with a finite number of orbits. Let $\Sigma$ be the corresponding fan. Then $\Sigma$ is simplicial. We denote by $\varepsilon_1, \varepsilon_2$ the positive rays and by $\tau_1, \ldots, \tau_k$ we denote the negative rays. Let $n_1, \ldots, n_k$ be the primitive vectors on the rays $\tau_1,\ldots, \tau_k$. We have $n_i = -(a_{1i}, a_{2i})$ for some relatively prime integers $a_{1i},a_{2i} \in \BZ_{\geq 0}.$ 
    
    Then $[D_{\tau_1}],\ldots, [D_{\tau_k}]$ is a basis of $\Cl(X) \simeq \BZ^k$ and 
    $$[D_{\varepsilon_{s}}] = \sum_i a_{si}[D_{\tau_i}]$$
    where $s = 1,2$. 

    Suppose there is no $i$ such that $a_{1i}$ and $a_{2i}$ are not equal to zero. Then $\Sigma$ is the fan of $\BP^1\times \BP^1.$

    Now suppose that there is $i$ such that $a_{1i}$ and $a_{2i}$ are not equal to zero. Consider the cone $\sigma = \tau_i.$ The set $\{[D_{\tau_j}]\mid j\neq i \} $, together with $[D_{\varepsilon_1}]$ and $[D_{\varepsilon_2}]$, generates $\Gamma(\sigma).$ The elements $\{[D_{\tau_j}]\mid j\neq i \} $ are irreducible in $\Gamma(\sigma)$ but $[D_{\varepsilon_1}]$ and $[D_{\varepsilon_2}]$ do not belong to the monoid generated by $\{[D_{\tau_j}]\mid j\neq i \}$. Then $[D_{\varepsilon_1}]$ or $[D_{\varepsilon_2}]$ is also irreducible in $\Gamma(\sigma).$ We will assume that it is $[D_{\varepsilon_1}].$ Since $\Gamma(\sigma)$ is a free monoid then the set of irreducible elements in $\Gamma(\sigma)$ is $\BQ$-linearly independent. Therefore, there are at most $k$ irreducible elements in $\Gamma(\sigma).$ So $[D_{\varepsilon_2}]$ belongs to the monoid generated by $[D_{\tau_j}]$ with $j\neq i$ and $[D_{\varepsilon_1}].$ So we have
    $$ 
    [D_{\varepsilon_2}] = d[D_{\varepsilon_1}] + \sum_{j\neq i}b_{j}[D_{\tau_j}]
    $$
    for $d, b_j \in \BZ_{\geq 0}.$ Let us also note that $d>0$, $a_{2j}  = da_{1j} + b_j$ for all $ j\neq i$  and $a_{2i} = da_{1i}.$ 
    
    Suppose we have another $s \neq i$ such that both $a_{1s}$ and $a_{2s}$ are not equal to zero. In the same way we get that either $[D_{\varepsilon_1}]$ or $[D_{\varepsilon_2}]$ is irreducible. All coordinates of $[D_{\varepsilon_2}]$ are greater or equal to the coordinates of $[D_{\varepsilon_1}]$ in the basis $[D_{\tau_1}], \ldots, [D_{\tau_k}].$ Hence, again $[D_{\varepsilon_1}]$ is irreducible. So we  obtain 
    $$[D_{\varepsilon_2}] = d'[D_{\varepsilon_1}] + \sum_{j\neq s}c_{j}[D_{\tau_j}]$$
    for $d', c_j \in \BZ_{\geq 0}.$ Therefore, $a_{2i} = da_{1i} = d'a_{1i} + c_i$ and $d \geq d'.$ But in the same way we will get that $d' \geq d.$ So $d = d'.$ Then 
    $$d = \frac{a_{1i}}{a_{2i}} = d' = \frac{a_{1s}}{a_{2s}}.$$
    Hence, the vectors $n_i$ and $n_s$ are proportional. Since they are primitive we get that $n_i = n_s$ and $i = s.$ It means that there is only one vector $n_i$ for which both coordinates are non-zero. For this vector we have $a_{2i} = da_{1i}$. The coordinates of vector $n_i$ are relatively prime, therefore $a_{1i} = 1$ and $a_{2i} = d.$ Further, we will assume that $i = 1.$

    If $k = 3$ and $n_2 = -(1,0), n_3 = -(0,1)$ or $k = 2, n_2 = -(1,0)$ and $d > 1$ then one can check that the monoid $\Gamma(\tau_1)$ is not free. In other cases the fan $\Sigma$ is the fan of weighted projective space $\BP(1,1,d)$ or Hirzebruch surface $\BF_d.$

    Direct check shows that all three types of varieties satisfy the condition of Theorem \ref{mainth}.
    
\end{proof}

    We now describe the orbits of the maximal unipotent group in the automorphism group for the varieties considered in this section.

    Let $X$ be a complete simpilicial toric variety. We will use the notations from Section \ref{prelim}.

    \begin{lemma}\label{lem6}
        Let $\sigma$ be a cone in $\Sigma$ and $O_{\sigma}$ the corresponding orbit of the torus $T$ in $X$. Then $\pi^{-1}(O_{\sigma}) = \{(x_{\rho}) \in \BA^d \mid x_{\rho} = 0 \iff \rho \in \sigma(1)\} \subseteq \widehat{X}.$ 
    \end{lemma}

    \begin{proof}
        The cone $\sigma$ corresponds to the open affine $T$-invariant chart $X_{\sigma} \subseteq X$ and $O_{\sigma}$ is the unique closed $T$-orbit in $X_{\sigma}$. Let $\sigma^{\vee}$ be the dual cone to $\sigma$ in the vector space $M_{\BQ}.$ Then
        $$\BK[X_{\sigma}] = \bigoplus_{m\in \sigma^{\vee}\cap M} \langle\chi^m\rangle,$$
        where $\chi^m \in \BK[T]$ is the character corresponding to $m$. The orbit $O_{\sigma}$ corresponds to the face $\tau$ of $\sigma^{\vee}$ where 
        $$\tau = \{m\in M_{\BQ}\cap \sigma^{\vee} \mid \langle m, n_{\rho}\rangle = 0 \ \forall \ \rho \in \sigma(1)  \}.$$
        The ideal corresponding to $O_{\sigma}$ in $\BK[X_{\sigma}]$ is 
        $$I(O_{\sigma}) = \bigoplus_{m\in M\cap(\sigma^{\vee}\setminus \tau)}\langle \chi^m \rangle.$$

        Let $\pi_{\sigma}$ be the restriction of $\pi$ to the subset
        $$V_{\sigma} = \{(x_{\rho_1},\ldots, x_{\rho_d)} \in \mathbb{A}^d \mid x_{\rho} \neq 0\ \forall \rho \notin \sigma(1)\}.$$
        Then the image of $\pi_{\sigma}$ is $X_{\sigma}$ and the dual homomorphism of algebras $\pi_{\sigma}^*: \BK[X_{\sigma}] \to \BK[V_{\sigma}]$ is given by the formula:
        $$\chi^m \to \prod_{\rho \in \Sigma(1)} x_{\rho}^{\langle m, n_{\rho}\rangle}.$$
        See Section 2 in \cite{COX2} for details. Therefore, the ideal $J = I(\pi_{\sigma}^{-1}(O_{\tau}))$ is the radical of the ideal generated by Laurent monomials $\prod_{\rho \in \Sigma(1)} x_{\rho}^{\langle m, n_{\rho}\rangle}$ with $m \in M\cap(\sigma^{\vee}\setminus \tau).$ Each of these monomials is divisible by some $x_{\rho}$ with $\rho \in \sigma(1).$ Therefore, $J$ is contained in the ideal generated by $x_{\rho}$ with $\rho \in \sigma(1).$ 
        
        Since $\sigma$ is simplicial for each $\rho_0 \in \sigma(1)$ there is a face $\eta$ in $\sigma^{\vee}$ such that $\langle m, n_{\rho} \rangle = 0$ for all $m\in \eta, \rho \in \sigma(1)\setminus \{\rho_0\}$ and there is $m'\in \eta$ with $\langle m', n_{\rho_0} \rangle \neq 0.$ Considering this $m'$ we obtain the monomial of the following form:
        $$x_{\rho_0}^{\langle m', n_{\rho_0}\rangle}\prod_{\rho \notin \Sigma(1)} x_{\rho}^{\langle m',  n_{\rho}\rangle}.$$
        The variables $x_{\rho}$ with $\rho \notin \sigma(1)$ are invertible on $V_{\sigma}.$ Hence, the monomial $x_{\rho_0}^{\langle m', n_{\rho_0}\rangle}$ belongs to $J$. Since $J$ is radical $x_{\rho_0} \in J.$ Therefore, $J$ is the ideal generated by $x_{\rho}$ with $\rho \in \sigma(1).$ It implies that
        $$V_{\sigma} \cap \pi^{-1}(O_{\sigma}) = \{(x_{\rho}) \in \BA^d \mid x_{\rho} = 0 \iff \rho \in \sigma(1)\}.$$
        
        Since $\pi$ is geometric quotient and $V_{\sigma}$ is $G_X$-invariant we have $\pi^{-1}(X_{\sigma}) = V_{\sigma}.$ Therefore, 
       $$ \pi^{-1}(O_{\sigma}) = \{(x_{\rho}) \in \BA^d \mid x_{\rho} = 0 \iff \rho \in \sigma(1)\}.$$
    \end{proof}
    
Now suppose $X = \BP(1,1,d_2,\ldots ,d_n)$ with $d_i\mid d_{i+1}.$ Then the set $\Sigma(1)$ consists of the rays:
$$\rho_0 = \tau = \langle (-1,-d_2, \ldots, -d_n)\rangle_{\geq 0}$$
$$\rho_i = \varepsilon_i = \langle (0, \ldots, 1, \ldots, 0)\rangle_{\geq 0}\ \text{for }\ i = 1,\ldots, n.$$
For any proper subset $P\subsetneq \{0,1,\ldots, n\}$ there is a cone $\sigma_P$ with $\sigma_P(1) = \{\rho_a \mid a \in P\}.$

It is not difficult to verify that all nonempty basic subsets in $\Sigma(1)$ contain exactly one element and for any $\rho_i$ the subset $\{\rho_i\}$ is basic. For a basic subset $A_i = \{\rho_i\}$ we have $\widehat{A_i} = \{\rho_j \mid j<i\}.$ Therefore, the corresponding set $Z_{A_i}$ has the following form:
$$Z_{A_i} = \{(x_{\rho_0}, \ldots, x_{\rho_n}) \mid  x_{\rho_0} = \ldots  = x_{\rho_{i-1}} = 0, \ x_{\rho_i} \neq 0 \}.$$
By Lemma \ref{lem3} the set $Z_{A_i}$ is a  $G_X\times U$-orbit and by Lemma \ref{lem4} the subset $Z_{A_i}$ is contained in  $\widehat{X}$. By Lemma \ref{lem6} the set $\pi(Z_{A_i})$ contains $T$-orbits $O_{\sigma}$ for all $\sigma\in \Sigma$ such that  $\rho_0,\ldots, \rho_{i-1}\in \sigma(1)$ and $\rho_i \notin \sigma(1)$. Therefore, we obtain the following proposition.

\begin{proposition}
There are $n+1$ orbits of a maximal unipotent group  on $\BP(1,1,d_2,\ldots ,d_n)$ with $d_i\mid d_{i+1}.$ If we choose the maximal unipotent subgroup $U$ as described in Section \ref{mainsec}, then for each $i = 0\ldots n$ there is an $U$-orbit $O_i$ which is the union of $T$-orbits $O_{\sigma}$ for cones $\sigma \in \Sigma$ with $\rho_0,\ldots, \rho_{i-1}\in \sigma(1)$ and $\rho_i \notin \sigma(1)$. 
\end{proposition}

Similarly, one can find $U$-orbits for Hirzebruch surface $\BF_d$ and $\BP^1\times \BP^1$. In Figure \ref{Fig5}, the cones corresponding to the $T$-orbits lying in the same $U$-orbit are drawn in the same color.

\begin{figure}[h]
\begin{center}
\input{FansOrbits}
\end{center}
\caption{The $U$-orbits on $\BP(1, 1, d),\ \BP^1\times \BP^1$ and $\BF_d$.}\label{Fig5}
\end{figure}

\end{document}

%% file: Picture.tex
\tikzset{every picture/.style={line width=0.75pt}} 

\begin{tikzpicture}[x=0.75pt,y=0.75pt,yscale=-1,xscale=1]

\draw [line width=1.5]    (490.33,149.33) -- (299.33,149.33) ;
\draw [line width=1.5]    (299.33,14.33) -- (299,151.17) ;
\draw [line width=1.5]    (106,245.17) -- (299,151.17) ;
\draw [line width=1.5]    (173,277.17) -- (299,151.17) ;
\draw [line width=1.5]    (299,151.17) -- (255.33,281) ;
\draw    (299,151.17) -- (324.34,149.47) ;
\draw [shift={(326.33,149.33)}, rotate = 176.16] [color={rgb, 255:red, 0; green, 0; blue, 0 }  ][line width=0.75]    (10.93,-3.29) .. controls (6.95,-1.4) and (3.31,-0.3) .. (0,0) .. controls (3.31,0.3) and (6.95,1.4) .. (10.93,3.29)   ;
\draw    (299,151.17) -- (299.31,124.33) ;
\draw [shift={(299.33,122.33)}, rotate = 90.66] [color={rgb, 255:red, 0; green, 0; blue, 0 }  ][line width=0.75]    (10.93,-3.29) .. controls (6.95,-1.4) and (3.31,-0.3) .. (0,0) .. controls (3.31,0.3) and (6.95,1.4) .. (10.93,3.29)   ;
\draw  [draw opacity=0] (56.33,14.33) -- (516.33,14.33) -- (516.33,285) -- (56.33,285) -- cycle ; \draw   (56.33,14.33) -- (56.33,285)(83.33,14.33) -- (83.33,285)(110.33,14.33) -- (110.33,285)(137.33,14.33) -- (137.33,285)(164.33,14.33) -- (164.33,285)(191.33,14.33) -- (191.33,285)(218.33,14.33) -- (218.33,285)(245.33,14.33) -- (245.33,285)(272.33,14.33) -- (272.33,285)(299.33,14.33) -- (299.33,285)(326.33,14.33) -- (326.33,285)(353.33,14.33) -- (353.33,285)(380.33,14.33) -- (380.33,285)(407.33,14.33) -- (407.33,285)(434.33,14.33) -- (434.33,285)(461.33,14.33) -- (461.33,285)(488.33,14.33) -- (488.33,285)(515.33,14.33) -- (515.33,285) ; \draw   (56.33,14.33) -- (516.33,14.33)(56.33,41.33) -- (516.33,41.33)(56.33,68.33) -- (516.33,68.33)(56.33,95.33) -- (516.33,95.33)(56.33,122.33) -- (516.33,122.33)(56.33,149.33) -- (516.33,149.33)(56.33,176.33) -- (516.33,176.33)(56.33,203.33) -- (516.33,203.33)(56.33,230.33) -- (516.33,230.33)(56.33,257.33) -- (516.33,257.33)(56.33,284.33) -- (516.33,284.33) ; \draw    ;

\draw (436.33,152.73) node [anchor=north west][inner sep=0.75pt]    {$\varepsilon _{1}$};
\draw (278,16.4) node [anchor=north west][inner sep=0.75pt]    {$\varepsilon _{2}$};
\draw (262,264.4) node [anchor=north west][inner sep=0.75pt]    {$\tau _{1}$};
\draw (184,264.4) node [anchor=north west][inner sep=0.75pt]    {$\tau _{2}$};
\draw (120,204.4) node [anchor=north west][inner sep=0.75pt]    {$\tau _{3}$};
\draw (311,157.4) node [anchor=north west][inner sep=0.75pt]    {$e_{1}$};
\draw (274,127.57) node [anchor=north west][inner sep=0.75pt]    {$e_{2}$};
\draw (409.33,44.73) node [anchor=north west][inner sep=0.75pt]    {$N_{\mathbb{Q}}$};

\end{tikzpicture}

%% file: Example2Fan.tex
\tikzset{every picture/.style={line width=0.75pt}} 

\begin{tikzpicture}[x=0.75pt,y=0.75pt,yscale=-1,xscale=1]

\draw  [draw opacity=0][line width=0.75]  (152.33,96) -- (481.33,96) -- (481.33,302) -- (152.33,302) -- cycle ; \draw  [color={rgb, 255:red, 0; green, 0; blue, 0 }  ,draw opacity=0.55 ][line width=0.75]  (152.33,96) -- (152.33,302)(193.33,96) -- (193.33,302)(234.33,96) -- (234.33,302)(275.33,96) -- (275.33,302)(316.33,96) -- (316.33,302)(357.33,96) -- (357.33,302)(398.33,96) -- (398.33,302)(439.33,96) -- (439.33,302)(480.33,96) -- (480.33,302) ; \draw  [color={rgb, 255:red, 0; green, 0; blue, 0 }  ,draw opacity=0.55 ][line width=0.75]  (152.33,96) -- (481.33,96)(152.33,137) -- (481.33,137)(152.33,178) -- (481.33,178)(152.33,219) -- (481.33,219)(152.33,260) -- (481.33,260)(152.33,301) -- (481.33,301) ; \draw  [color={rgb, 255:red, 0; green, 0; blue, 0 }  ,draw opacity=0.55 ][line width=0.75]   ;
\draw [line width=1.5]    (316.33,96) -- (316.33,178.33) ;
\draw [line width=1.5]    (480.33,178) -- (316.33,178.33) ;
\draw [line width=1.5]    (316.33,301.33) -- (316.33,178.33) ;
\draw [line width=1.5]    (316.33,178.33) -- (256.33,300) ;
\draw [line width=2.25]    (316.33,178.33) -- (353.33,178.33) ;
\draw [shift={(357.33,178.33)}, rotate = 180] [color={rgb, 255:red, 0; green, 0; blue, 0 }  ][line width=2.25]    (17.49,-5.26) .. controls (11.12,-2.23) and (5.29,-0.48) .. (0,0) .. controls (5.29,0.48) and (11.12,2.23) .. (17.49,5.26)   ;
\draw [line width=2.25]    (316.33,178.33) -- (316.33,141.33) ;
\draw [shift={(316.33,137.33)}, rotate = 90] [color={rgb, 255:red, 0; green, 0; blue, 0 }  ][line width=2.25]    (17.49,-5.26) .. controls (11.12,-2.23) and (5.29,-0.48) .. (0,0) .. controls (5.29,0.48) and (11.12,2.23) .. (17.49,5.26)   ;
\draw [line width=2.25]    (316.33,178.33) -- (316.33,215.33) ;
\draw [shift={(316.33,219.33)}, rotate = 270] [color={rgb, 255:red, 0; green, 0; blue, 0 }  ][line width=2.25]    (17.49,-5.26) .. controls (11.12,-2.23) and (5.29,-0.48) .. (0,0) .. controls (5.29,0.48) and (11.12,2.23) .. (17.49,5.26)   ;
\draw [line width=2.25]    (316.33,178.33) -- (277.12,256.76) ;
\draw [shift={(275.33,260.33)}, rotate = 296.57] [color={rgb, 255:red, 0; green, 0; blue, 0 }  ][line width=2.25]    (17.49,-5.26) .. controls (11.12,-2.23) and (5.29,-0.48) .. (0,0) .. controls (5.29,0.48) and (11.12,2.23) .. (17.49,5.26)   ;

\draw (448,105.4) node [anchor=north west][inner sep=0.75pt]    {$N_{\mathbb{Q}}$};
\draw (359.33,181.73) node [anchor=north west][inner sep=0.75pt]    {$e_{1}$};
\draw (288.33,140.73) node [anchor=north west][inner sep=0.75pt]    {$e_{2}$};
\draw (318.33,222.73) node [anchor=north west][inner sep=0.75pt]    {$n_{2}$};
\draw (277.33,263.73) node [anchor=north west][inner sep=0.75pt]    {$n_{1}$};
\draw (456,181.4) node [anchor=north west][inner sep=0.75pt]    {$ \begin{array}{l}
\varepsilon _{1}\\
\end{array}$};
\draw (292,96.4) node [anchor=north west][inner sep=0.75pt]    {$ \begin{array}{l}
\varepsilon _{2}\\
\end{array}$};
\draw (236,274.4) node [anchor=north west][inner sep=0.75pt]    {$ \begin{array}{l}
\tau _{1}\\
\end{array}$};
\draw (320,270.4) node [anchor=north west][inner sep=0.75pt]    {$ \begin{array}{l}
\tau _{2}\\
\end{array}$};

\end{tikzpicture}

%% file: Example2Cl.tex
\tikzset{every picture/.style={line width=0.75pt}} 

\begin{tikzpicture}[x=0.75pt,y=0.75pt,yscale=-1,xscale=1]

\draw  [draw opacity=0] (182.33,51) -- (472.33,51) -- (472.33,231.83) -- (182.33,231.83) -- cycle ; \draw  [color={rgb, 255:red, 128; green, 128; blue, 128 }  ,draw opacity=1 ] (182.33,51) -- (182.33,231.83)(218.33,51) -- (218.33,231.83)(254.33,51) -- (254.33,231.83)(290.33,51) -- (290.33,231.83)(326.33,51) -- (326.33,231.83)(362.33,51) -- (362.33,231.83)(398.33,51) -- (398.33,231.83)(434.33,51) -- (434.33,231.83)(470.33,51) -- (470.33,231.83) ; \draw  [color={rgb, 255:red, 128; green, 128; blue, 128 }  ,draw opacity=1 ] (182.33,51) -- (472.33,51)(182.33,87) -- (472.33,87)(182.33,123) -- (472.33,123)(182.33,159) -- (472.33,159)(182.33,195) -- (472.33,195)(182.33,231) -- (472.33,231) ; \draw  [color={rgb, 255:red, 128; green, 128; blue, 128 }  ,draw opacity=1 ]  ;
\draw  [fill={rgb, 255:red, 0; green, 0; blue, 0 }  ,fill opacity=1 ] (322,159) .. controls (322,156.84) and (323.79,155.08) .. (326,155.08) .. controls (328.21,155.08) and (330,156.84) .. (330,159) .. controls (330,161.16) and (328.21,162.92) .. (326,162.92) .. controls (323.79,162.92) and (322,161.16) .. (322,159) -- cycle ;
\draw  [fill={rgb, 255:red, 0; green, 0; blue, 0 }  ,fill opacity=1 ] (286,123) .. controls (286,120.84) and (287.79,119.08) .. (290,119.08) .. controls (292.21,119.08) and (294,120.84) .. (294,123) .. controls (294,125.16) and (292.21,126.92) .. (290,126.92) .. controls (287.79,126.92) and (286,125.16) .. (286,123) -- cycle ;
\draw  [fill={rgb, 255:red, 0; green, 0; blue, 0 }  ,fill opacity=1 ] (358,123) .. controls (358,120.84) and (359.79,119.08) .. (362,119.08) .. controls (364.21,119.08) and (366,120.84) .. (366,123) .. controls (366,125.16) and (364.21,126.92) .. (362,126.92) .. controls (359.79,126.92) and (358,125.16) .. (358,123) -- cycle ;
\draw    (290,51) -- (290,159) ;
\draw    (290,159) -- (470.33,159) ;

\draw (279,159.4) node [anchor=north west][inner sep=0.75pt]    {$0$};
\draw (328,162.4) node [anchor=north west][inner sep=0.75pt]    {$[ D_{\tau _{1}}] \ =\ [ D_{\varepsilon _{1}}]$};
\draw (243,128.4) node [anchor=north west][inner sep=0.75pt]    {$[ D_{\tau _{2}}]$};
\draw (364,126.4) node [anchor=north west][inner sep=0.75pt]    {$[ D_{\varepsilon _{2}}]$};
\draw (396,66.4) node [anchor=north west][inner sep=0.75pt]    {$\mathrm{Cl}( X)$};

\end{tikzpicture}

%% file: Surfaces.tex
\tikzset{every picture/.style={line width=0.75pt}} 

\begin{tikzpicture}[x=0.75pt,y=0.75pt,yscale=-1,xscale=1]

\draw  [draw opacity=0] (17,42) -- (185.33,42) -- (185.33,211) -- (17,211) -- cycle ; \draw  [color={rgb, 255:red, 0; green, 0; blue, 0 }  ,draw opacity=0.6 ] (17,42) -- (17,211)(45,42) -- (45,211)(73,42) -- (73,211)(101,42) -- (101,211)(129,42) -- (129,211)(157,42) -- (157,211)(185,42) -- (185,211) ; \draw  [color={rgb, 255:red, 0; green, 0; blue, 0 }  ,draw opacity=0.6 ] (17,42) -- (185.33,42)(17,70) -- (185.33,70)(17,98) -- (185.33,98)(17,126) -- (185.33,126)(17,154) -- (185.33,154)(17,182) -- (185.33,182)(17,210) -- (185.33,210) ; \draw  [color={rgb, 255:red, 0; green, 0; blue, 0 }  ,draw opacity=0.6 ]  ;
\draw  [draw opacity=0] (214,41) -- (382.33,41) -- (382.33,210) -- (214,210) -- cycle ; \draw  [color={rgb, 255:red, 0; green, 0; blue, 0 }  ,draw opacity=0.6 ] (214,41) -- (214,210)(242,41) -- (242,210)(270,41) -- (270,210)(298,41) -- (298,210)(326,41) -- (326,210)(354,41) -- (354,210)(382,41) -- (382,210) ; \draw  [color={rgb, 255:red, 0; green, 0; blue, 0 }  ,draw opacity=0.6 ] (214,41) -- (382.33,41)(214,69) -- (382.33,69)(214,97) -- (382.33,97)(214,125) -- (382.33,125)(214,153) -- (382.33,153)(214,181) -- (382.33,181)(214,209) -- (382.33,209) ; \draw  [color={rgb, 255:red, 0; green, 0; blue, 0 }  ,draw opacity=0.6 ]  ;
\draw  [draw opacity=0] (414,40) -- (582.33,40) -- (582.33,209) -- (414,209) -- cycle ; \draw  [color={rgb, 255:red, 0; green, 0; blue, 0 }  ,draw opacity=0.6 ] (414,40) -- (414,209)(442,40) -- (442,209)(470,40) -- (470,209)(498,40) -- (498,209)(526,40) -- (526,209)(554,40) -- (554,209)(582,40) -- (582,209) ; \draw  [color={rgb, 255:red, 0; green, 0; blue, 0 }  ,draw opacity=0.6 ] (414,40) -- (582.33,40)(414,68) -- (582.33,68)(414,96) -- (582.33,96)(414,124) -- (582.33,124)(414,152) -- (582.33,152)(414,180) -- (582.33,180)(414,208) -- (582.33,208) ; \draw  [color={rgb, 255:red, 0; green, 0; blue, 0 }  ,draw opacity=0.6 ]  ;
\draw [line width=1.5]    (100,125) -- (101,42) ;
\draw [line width=1.5]    (101,126) -- (185,126) ;
\draw [line width=1.5]    (101,126) -- (58.33,209) ;
\draw [line width=1.5]    (101,126) -- (126,126) ;
\draw [shift={(129,126)}, rotate = 180] [color={rgb, 255:red, 0; green, 0; blue, 0 }  ][line width=1.5]    (14.21,-4.28) .. controls (9.04,-1.82) and (4.3,-0.39) .. (0,0) .. controls (4.3,0.39) and (9.04,1.82) .. (14.21,4.28)   ;
\draw [line width=1.5]    (101,126) -- (101,101) ;
\draw [shift={(101,98)}, rotate = 90] [color={rgb, 255:red, 0; green, 0; blue, 0 }  ][line width=1.5]    (14.21,-4.28) .. controls (9.04,-1.82) and (4.3,-0.39) .. (0,0) .. controls (4.3,0.39) and (9.04,1.82) .. (14.21,4.28)   ;
\draw [line width=1.5]    (101,126) -- (74.34,179.32) ;
\draw [shift={(73,182)}, rotate = 296.57] [color={rgb, 255:red, 0; green, 0; blue, 0 }  ][line width=1.5]    (14.21,-4.28) .. controls (9.04,-1.82) and (4.3,-0.39) .. (0,0) .. controls (4.3,0.39) and (9.04,1.82) .. (14.21,4.28)   ;
\draw [line width=1.5]    (298,125) -- (382,125) ;
\draw [line width=1.5]    (298,41) -- (298,125) ;
\draw [line width=1.5]    (298,209) -- (298,125) ;
\draw [line width=1.5]    (214,125) -- (298,125) ;
\draw [line width=1.5]    (298,125) -- (323,125) ;
\draw [shift={(326,125)}, rotate = 180] [color={rgb, 255:red, 0; green, 0; blue, 0 }  ][line width=1.5]    (14.21,-4.28) .. controls (9.04,-1.82) and (4.3,-0.39) .. (0,0) .. controls (4.3,0.39) and (9.04,1.82) .. (14.21,4.28)   ;
\draw [line width=1.5]    (298,125) -- (298,100) ;
\draw [shift={(298,97)}, rotate = 90] [color={rgb, 255:red, 0; green, 0; blue, 0 }  ][line width=1.5]    (14.21,-4.28) .. controls (9.04,-1.82) and (4.3,-0.39) .. (0,0) .. controls (4.3,0.39) and (9.04,1.82) .. (14.21,4.28)   ;
\draw [line width=1.5]    (298,125) -- (298,150) ;
\draw [shift={(298,153)}, rotate = 270] [color={rgb, 255:red, 0; green, 0; blue, 0 }  ][line width=1.5]    (14.21,-4.28) .. controls (9.04,-1.82) and (4.3,-0.39) .. (0,0) .. controls (4.3,0.39) and (9.04,1.82) .. (14.21,4.28)   ;
\draw [line width=1.5]    (298,125) -- (273,125) ;
\draw [shift={(270,125)}, rotate = 360] [color={rgb, 255:red, 0; green, 0; blue, 0 }  ][line width=1.5]    (14.21,-4.28) .. controls (9.04,-1.82) and (4.3,-0.39) .. (0,0) .. controls (4.3,0.39) and (9.04,1.82) .. (14.21,4.28)   ;
\draw [line width=1.5]    (498,124) -- (582,124) ;
\draw [line width=1.5]    (498,40) -- (498,124) ;
\draw [line width=1.5]    (498,208) -- (498,124) ;
\draw [line width=1.5]    (455.33,208) -- (498,124) ;
\draw [line width=1.5]    (498,124) -- (523,124) ;
\draw [shift={(526,124)}, rotate = 180] [color={rgb, 255:red, 0; green, 0; blue, 0 }  ][line width=1.5]    (14.21,-4.28) .. controls (9.04,-1.82) and (4.3,-0.39) .. (0,0) .. controls (4.3,0.39) and (9.04,1.82) .. (14.21,4.28)   ;
\draw [line width=1.5]    (498,124) -- (498,99) ;
\draw [shift={(498,96)}, rotate = 90] [color={rgb, 255:red, 0; green, 0; blue, 0 }  ][line width=1.5]    (14.21,-4.28) .. controls (9.04,-1.82) and (4.3,-0.39) .. (0,0) .. controls (4.3,0.39) and (9.04,1.82) .. (14.21,4.28)   ;
\draw [line width=1.5]    (498,124) -- (498,149) ;
\draw [shift={(498,152)}, rotate = 270] [color={rgb, 255:red, 0; green, 0; blue, 0 }  ][line width=1.5]    (14.21,-4.28) .. controls (9.04,-1.82) and (4.3,-0.39) .. (0,0) .. controls (4.3,0.39) and (9.04,1.82) .. (14.21,4.28)   ;
\draw [line width=1.5]    (498,124) -- (471.34,177.32) ;
\draw [shift={(470,180)}, rotate = 296.57] [color={rgb, 255:red, 0; green, 0; blue, 0 }  ][line width=1.5]    (14.21,-4.28) .. controls (9.04,-1.82) and (4.3,-0.39) .. (0,0) .. controls (4.3,0.39) and (9.04,1.82) .. (14.21,4.28)   ;

\draw (103,129.4) node [anchor=north west][inner sep=0.75pt]    {$e_{1}$};
\draw (75,101.4) node [anchor=north west][inner sep=0.75pt]    {$e_{2}$};
\draw (75,185.4) node [anchor=north west][inner sep=0.75pt]    {$n_{1}( -1,-d)$};
\draw (314,128.4) node [anchor=north west][inner sep=0.75pt]    {$e_{1}$};
\draw (278,94.4) node [anchor=north west][inner sep=0.75pt]    {$e_{2}$};
\draw (528,127.4) node [anchor=north west][inner sep=0.75pt]    {$e_{1}$};
\draw (476,90.4) node [anchor=north west][inner sep=0.75pt]    {$e_{2}$};
\draw (300,156.4) node [anchor=north west][inner sep=0.75pt]    {$n_{2}$};
\draw (272,128.4) node [anchor=north west][inner sep=0.75pt]    {$n_{1}$};
\draw (500,141.4) node [anchor=north west][inner sep=0.75pt]    {$n_{2}$};
\draw (472,183.4) node [anchor=north west][inner sep=0.75pt]    {$n_{1}( -1,-d)$};
\draw (45,179.4) node [anchor=north west][inner sep=0.75pt]    {$\tau _{1}$};
\draw (216,128.4) node [anchor=north west][inner sep=0.75pt]    {$\tau _{1}$};
\draw (438,188.4) node [anchor=north west][inner sep=0.75pt]    {$\tau _{1}$};
\draw (300,184.4) node [anchor=north west][inner sep=0.75pt]    {$\tau _{2}$};
\draw (500,161.4) node [anchor=north west][inner sep=0.75pt]    {$\tau _{2}$};
\draw (159,129.4) node [anchor=north west][inner sep=0.75pt]    {$ \begin{array}{l}
\varepsilon _{1}\\
\end{array}$};
\draw (558,122.4) node [anchor=north west][inner sep=0.75pt]    {$ \begin{array}{l}
\varepsilon _{1}\\
\end{array}$};
\draw (103,45.4) node [anchor=north west][inner sep=0.75pt]    {$ \begin{array}{l}
\varepsilon _{2}\\
\end{array}$};
\draw (300,44.4) node [anchor=north west][inner sep=0.75pt]    {$ \begin{array}{l}
\varepsilon _{2}\\
\end{array}$};
\draw (500,43.4) node [anchor=north west][inner sep=0.75pt]    {$ \begin{array}{l}
\varepsilon _{2}\\
\end{array}$};

\end{tikzpicture}

%% file: FansOrbits.tex
\tikzset{every picture/.style={line width=0.75pt}} 

\begin{tikzpicture}[x=0.75pt,y=0.75pt,yscale=-1,xscale=1]

\draw  [draw opacity=0] (36,61) -- (204.33,61) -- (204.33,230) -- (36,230) -- cycle ; \draw  [color={rgb, 255:red, 0; green, 0; blue, 0 }  ,draw opacity=0.6 ] (36,61) -- (36,230)(64,61) -- (64,230)(92,61) -- (92,230)(120,61) -- (120,230)(148,61) -- (148,230)(176,61) -- (176,230)(204,61) -- (204,230) ; \draw  [color={rgb, 255:red, 0; green, 0; blue, 0 }  ,draw opacity=0.6 ] (36,61) -- (204.33,61)(36,89) -- (204.33,89)(36,117) -- (204.33,117)(36,145) -- (204.33,145)(36,173) -- (204.33,173)(36,201) -- (204.33,201)(36,229) -- (204.33,229) ; \draw  [color={rgb, 255:red, 0; green, 0; blue, 0 }  ,draw opacity=0.6 ]  ;
\draw  [draw opacity=0] (234,61) -- (402.33,61) -- (402.33,230) -- (234,230) -- cycle ; \draw  [color={rgb, 255:red, 0; green, 0; blue, 0 }  ,draw opacity=0.6 ] (234,61) -- (234,230)(262,61) -- (262,230)(290,61) -- (290,230)(318,61) -- (318,230)(346,61) -- (346,230)(374,61) -- (374,230)(402,61) -- (402,230) ; \draw  [color={rgb, 255:red, 0; green, 0; blue, 0 }  ,draw opacity=0.6 ] (234,61) -- (402.33,61)(234,89) -- (402.33,89)(234,117) -- (402.33,117)(234,145) -- (402.33,145)(234,173) -- (402.33,173)(234,201) -- (402.33,201)(234,229) -- (402.33,229) ; \draw  [color={rgb, 255:red, 0; green, 0; blue, 0 }  ,draw opacity=0.6 ]  ;
\draw  [draw opacity=0] (434,62) -- (602.33,62) -- (602.33,231) -- (434,231) -- cycle ; \draw  [color={rgb, 255:red, 0; green, 0; blue, 0 }  ,draw opacity=0.6 ] (434,62) -- (434,231)(462,62) -- (462,231)(490,62) -- (490,231)(518,62) -- (518,231)(546,62) -- (546,231)(574,62) -- (574,231)(602,62) -- (602,231) ; \draw  [color={rgb, 255:red, 0; green, 0; blue, 0 }  ,draw opacity=0.6 ] (434,62) -- (602.33,62)(434,90) -- (602.33,90)(434,118) -- (602.33,118)(434,146) -- (602.33,146)(434,174) -- (602.33,174)(434,202) -- (602.33,202)(434,230) -- (602.33,230) ; \draw  [color={rgb, 255:red, 0; green, 0; blue, 0 }  ,draw opacity=0.6 ]  ;
\draw [color={rgb, 255:red, 208; green, 2; blue, 27 }  ,draw opacity=1 ][line width=1.5]    (120,145) -- (121,62) ;
\draw [color={rgb, 255:red, 208; green, 2; blue, 27 }  ,draw opacity=1 ][line width=1.5]    (121,146) -- (205,146) ;
\draw [color={rgb, 255:red, 65; green, 117; blue, 5 }  ,draw opacity=1 ][line width=1.5]    (121,146) -- (78.33,229) ;
\draw [color={rgb, 255:red, 208; green, 2; blue, 27 }  ,draw opacity=1 ][line width=1.5]    (121,146) -- (146,146) ;
\draw [shift={(149,146)}, rotate = 180] [color={rgb, 255:red, 208; green, 2; blue, 27 }  ,draw opacity=1 ][line width=1.5]    (14.21,-4.28) .. controls (9.04,-1.82) and (4.3,-0.39) .. (0,0) .. controls (4.3,0.39) and (9.04,1.82) .. (14.21,4.28)   ;
\draw [color={rgb, 255:red, 208; green, 2; blue, 27 }  ,draw opacity=1 ][line width=1.5]    (121,146) -- (121,121) ;
\draw [shift={(121,118)}, rotate = 90] [color={rgb, 255:red, 208; green, 2; blue, 27 }  ,draw opacity=1 ][line width=1.5]    (14.21,-4.28) .. controls (9.04,-1.82) and (4.3,-0.39) .. (0,0) .. controls (4.3,0.39) and (9.04,1.82) .. (14.21,4.28)   ;
\draw [color={rgb, 255:red, 65; green, 117; blue, 5 }  ,draw opacity=1 ][line width=1.5]    (121,146) -- (94.34,199.32) ;
\draw [shift={(93,202)}, rotate = 296.57] [color={rgb, 255:red, 65; green, 117; blue, 5 }  ,draw opacity=1 ][line width=1.5]    (14.21,-4.28) .. controls (9.04,-1.82) and (4.3,-0.39) .. (0,0) .. controls (4.3,0.39) and (9.04,1.82) .. (14.21,4.28)   ;
\draw [color={rgb, 255:red, 208; green, 2; blue, 27 }  ,draw opacity=1 ][line width=1.5]    (318,145) -- (402,145) ;
\draw [color={rgb, 255:red, 208; green, 2; blue, 27 }  ,draw opacity=1 ][line width=1.5]    (318,61) -- (318,145) ;
\draw [color={rgb, 255:red, 74; green, 144; blue, 226 }  ,draw opacity=1 ][line width=1.5]    (318,229) -- (318,145) ;
\draw [color={rgb, 255:red, 65; green, 117; blue, 5 }  ,draw opacity=1 ][line width=1.5]    (234,145) -- (253,145) -- (318,145) ;
\draw [color={rgb, 255:red, 208; green, 2; blue, 27 }  ,draw opacity=1 ][line width=1.5]    (318,145) -- (343,145) ;
\draw [shift={(346,145)}, rotate = 180] [color={rgb, 255:red, 208; green, 2; blue, 27 }  ,draw opacity=1 ][line width=1.5]    (14.21,-4.28) .. controls (9.04,-1.82) and (4.3,-0.39) .. (0,0) .. controls (4.3,0.39) and (9.04,1.82) .. (14.21,4.28)   ;
\draw [color={rgb, 255:red, 208; green, 2; blue, 27 }  ,draw opacity=1 ][line width=1.5]    (318,145) -- (318,120) ;
\draw [shift={(318,117)}, rotate = 90] [color={rgb, 255:red, 208; green, 2; blue, 27 }  ,draw opacity=1 ][line width=1.5]    (14.21,-4.28) .. controls (9.04,-1.82) and (4.3,-0.39) .. (0,0) .. controls (4.3,0.39) and (9.04,1.82) .. (14.21,4.28)   ;
\draw [color={rgb, 255:red, 74; green, 99; blue, 226 }  ,draw opacity=1 ][line width=1.5]    (318,145) -- (318,170) ;
\draw [shift={(318,173)}, rotate = 270] [color={rgb, 255:red, 74; green, 99; blue, 226 }  ,draw opacity=1 ][line width=1.5]    (14.21,-4.28) .. controls (9.04,-1.82) and (4.3,-0.39) .. (0,0) .. controls (4.3,0.39) and (9.04,1.82) .. (14.21,4.28)   ;
\draw [color={rgb, 255:red, 65; green, 117; blue, 5 }  ,draw opacity=1 ][line width=1.5]    (318,145) -- (293,145) ;
\draw [shift={(290,145)}, rotate = 360] [color={rgb, 255:red, 65; green, 117; blue, 5 }  ,draw opacity=1 ][line width=1.5]    (14.21,-4.28) .. controls (9.04,-1.82) and (4.3,-0.39) .. (0,0) .. controls (4.3,0.39) and (9.04,1.82) .. (14.21,4.28)   ;
\draw [color={rgb, 255:red, 208; green, 2; blue, 27 }  ,draw opacity=1 ][fill={rgb, 255:red, 208; green, 2; blue, 27 }  ,fill opacity=1 ][line width=1.5]    (518,146) -- (602,146) ;
\draw [color={rgb, 255:red, 208; green, 2; blue, 27 }  ,draw opacity=1 ][line width=1.5]    (518,60) -- (518,144) ;
\draw [color={rgb, 255:red, 74; green, 144; blue, 226 }  ,draw opacity=1 ][line width=1.5]    (518,228) -- (518,144) ;
\draw [color={rgb, 255:red, 65; green, 117; blue, 5 }  ,draw opacity=1 ][line width=1.5]    (475.33,231) -- (518,144) ;
\draw [color={rgb, 255:red, 208; green, 2; blue, 27 }  ,draw opacity=1 ][line width=1.5]    (518,146) -- (532,147) -- (543.01,146.21) ;
\draw [shift={(546,146)}, rotate = 175.91] [color={rgb, 255:red, 208; green, 2; blue, 27 }  ,draw opacity=1 ][line width=1.5]    (14.21,-4.28) .. controls (9.04,-1.82) and (4.3,-0.39) .. (0,0) .. controls (4.3,0.39) and (9.04,1.82) .. (14.21,4.28)   ;
\draw [color={rgb, 255:red, 208; green, 2; blue, 27 }  ,draw opacity=1 ][line width=1.5]    (518,144) -- (518,119) ;
\draw [shift={(518,116)}, rotate = 90] [color={rgb, 255:red, 208; green, 2; blue, 27 }  ,draw opacity=1 ][line width=1.5]    (14.21,-4.28) .. controls (9.04,-1.82) and (4.3,-0.39) .. (0,0) .. controls (4.3,0.39) and (9.04,1.82) .. (14.21,4.28)   ;
\draw [color={rgb, 255:red, 74; green, 144; blue, 226 }  ,draw opacity=1 ][line width=1.5]    (518,144) -- (518,169) ;
\draw [shift={(518,172)}, rotate = 270] [color={rgb, 255:red, 74; green, 144; blue, 226 }  ,draw opacity=1 ][line width=1.5]    (14.21,-4.28) .. controls (9.04,-1.82) and (4.3,-0.39) .. (0,0) .. controls (4.3,0.39) and (9.04,1.82) .. (14.21,4.28)   ;
\draw [color={rgb, 255:red, 65; green, 117; blue, 5 }  ,draw opacity=1 ][line width=1.5]    (518,144) -- (491.34,197.32) ;
\draw [shift={(490,200)}, rotate = 296.57] [color={rgb, 255:red, 65; green, 117; blue, 5 }  ,draw opacity=1 ][line width=1.5]    (14.21,-4.28) .. controls (9.04,-1.82) and (4.3,-0.39) .. (0,0) .. controls (4.3,0.39) and (9.04,1.82) .. (14.21,4.28)   ;
\draw  [color={rgb, 255:red, 208; green, 2; blue, 27 }  ,draw opacity=0 ][fill={rgb, 255:red, 208; green, 2; blue, 27 }  ,fill opacity=0.42 ] (518,62) -- (602,62) -- (602,146) -- (518,146) -- cycle ;
\draw  [color={rgb, 255:red, 74; green, 144; blue, 226 }  ,draw opacity=0 ][fill={rgb, 255:red, 74; green, 144; blue, 226 }  ,fill opacity=0.57 ] (518,146) -- (602,146) -- (602,230) -- (518,230) -- cycle ;
\draw  [color={rgb, 255:red, 0; green, 0; blue, 0 }  ,draw opacity=0 ][fill={rgb, 255:red, 189; green, 16; blue, 224 }  ,fill opacity=0.6 ] (475.33,230) -- (518,146) -- (518,230) -- cycle ;
\draw  [color={rgb, 255:red, 0; green, 0; blue, 0 }  ,draw opacity=0 ][fill={rgb, 255:red, 184; green, 233; blue, 134 }  ,fill opacity=0.6 ] (434,62) -- (518,62) -- (518,146) -- (434,146) -- cycle ;
\draw  [color={rgb, 255:red, 0; green, 0; blue, 0 }  ,draw opacity=0 ][fill={rgb, 255:red, 184; green, 233; blue, 134 }  ,fill opacity=0.6 ] (433.33,146) -- (475.33,146) -- (475.33,231) -- (433.33,231) -- cycle ;
\draw  [color={rgb, 255:red, 0; green, 0; blue, 0 }  ,draw opacity=0 ][fill={rgb, 255:red, 184; green, 233; blue, 134 }  ,fill opacity=0.6 ] (515.33,146) -- (475.33,228) -- (475.33,146) -- cycle ;
\draw  [fill={rgb, 255:red, 208; green, 2; blue, 27 }  ,fill opacity=1 ] (512,146) .. controls (512,142.69) and (514.69,140) .. (518,140) .. controls (521.31,140) and (524,142.69) .. (524,146) .. controls (524,149.31) and (521.31,152) .. (518,152) .. controls (514.69,152) and (512,149.31) .. (512,146) -- cycle ;
\draw  [color={rgb, 255:red, 208; green, 2; blue, 27 }  ,draw opacity=0 ][fill={rgb, 255:red, 208; green, 2; blue, 27 }  ,fill opacity=0.42 ] (318,61) -- (402,61) -- (402,145) -- (318,145) -- cycle ;
\draw  [color={rgb, 255:red, 184; green, 233; blue, 134 }  ,draw opacity=0 ][fill={rgb, 255:red, 184; green, 233; blue, 134 }  ,fill opacity=0.6 ] (234,61) -- (318,61) -- (318,145) -- (234,145) -- cycle ;
\draw  [color={rgb, 255:red, 74; green, 144; blue, 226 }  ,draw opacity=0 ][fill={rgb, 255:red, 74; green, 144; blue, 226 }  ,fill opacity=0.6 ] (318,145) -- (402,145) -- (402,230) -- (318,230) -- cycle ;
\draw  [color={rgb, 255:red, 208; green, 2; blue, 27 }  ,draw opacity=0 ][fill={rgb, 255:red, 189; green, 16; blue, 224 }  ,fill opacity=0.6 ] (234,145) -- (318,145) -- (318,230) -- (234,230) -- cycle ;
\draw  [fill={rgb, 255:red, 208; green, 2; blue, 27 }  ,fill opacity=1 ] (312,143) .. controls (312,139.69) and (314.69,137) .. (318,137) .. controls (321.31,137) and (324,139.69) .. (324,143) .. controls (324,146.31) and (321.31,149) .. (318,149) .. controls (314.69,149) and (312,146.31) .. (312,143) -- cycle ;
\draw  [color={rgb, 255:red, 208; green, 2; blue, 27 }  ,draw opacity=0 ][fill={rgb, 255:red, 208; green, 2; blue, 27 }  ,fill opacity=0.42 ] (121,61) -- (205,61) -- (205,146) -- (121,146) -- cycle ;
\draw  [color={rgb, 255:red, 184; green, 233; blue, 134 }  ,draw opacity=0 ][fill={rgb, 255:red, 184; green, 233; blue, 134 }  ,fill opacity=0.6 ] (36,61) -- (120,61) -- (120,145) -- (36,145) -- cycle ;
\draw  [color={rgb, 255:red, 184; green, 233; blue, 134 }  ,draw opacity=0 ][fill={rgb, 255:red, 184; green, 233; blue, 134 }  ,fill opacity=0.6 ] (36,146) -- (79,146) -- (79,230) -- (36,230) -- cycle ;
\draw  [color={rgb, 255:red, 0; green, 0; blue, 0 }  ,draw opacity=0 ][fill={rgb, 255:red, 184; green, 233; blue, 134 }  ,fill opacity=0.6 ] (119,146) -- (79,228) -- (79,146) -- cycle ;
\draw  [color={rgb, 255:red, 74; green, 144; blue, 226 }  ,draw opacity=0 ][fill={rgb, 255:red, 74; green, 144; blue, 226 }  ,fill opacity=0.6 ] (119,145) -- (204,145) -- (204,230) -- (119,230) -- cycle ;
\draw  [color={rgb, 255:red, 0; green, 0; blue, 0 }  ,draw opacity=0 ][fill={rgb, 255:red, 74; green, 144; blue, 226 }  ,fill opacity=0.6 ] (76.33,229) -- (119,146) -- (119,229) -- cycle ;
\draw  [fill={rgb, 255:red, 208; green, 2; blue, 27 }  ,fill opacity=1 ] (115,146) .. controls (115,142.69) and (117.69,140) .. (121,140) .. controls (124.31,140) and (127,142.69) .. (127,146) .. controls (127,149.31) and (124.31,152) .. (121,152) .. controls (117.69,152) and (115,149.31) .. (115,146) -- cycle ;

\draw (65,199.4) node [anchor=north west][inner sep=0.75pt]    {$\tau _{1}$};
\draw (236,148.4) node [anchor=north west][inner sep=0.75pt]    {$\tau _{1}$};
\draw (458,208.4) node [anchor=north west][inner sep=0.75pt]    {$\tau _{1}$};
\draw (320,204.4) node [anchor=north west][inner sep=0.75pt]    {$\tau _{2}$};
\draw (520,203.4) node [anchor=north west][inner sep=0.75pt]    {$\tau _{2}$};
\draw (179,149.4) node [anchor=north west][inner sep=0.75pt]    {$ \begin{array}{l}
\varepsilon _{1}\\
\end{array}$};
\draw (576,119.4) node [anchor=north west][inner sep=0.75pt]    {$ \begin{array}{l}
\varepsilon _{1}\\
\end{array}$};
\draw (123,65.4) node [anchor=north west][inner sep=0.75pt]    {$ \begin{array}{l}
\varepsilon _{2}\\
\end{array}$};
\draw (320,64.4) node [anchor=north west][inner sep=0.75pt]    {$ \begin{array}{l}
\varepsilon _{2}\\
\end{array}$};
\draw (520,63.4) node [anchor=north west][inner sep=0.75pt]    {$ \begin{array}{l}
\varepsilon _{2}\\
\end{array}$};

\end{tikzpicture}